\documentclass[final,a4paper]{siamltex}
\usepackage{amsmath,amssymb,amsfonts}
\usepackage{color}
\usepackage{enumerate}
\usepackage{verbatim}
\usepackage{eepic,epic,epsfig,epstopdf}
\usepackage[]{graphics}
\usepackage{graphicx,inputenc}
\usepackage{subfigure}
\usepackage{changepage}
\usepackage{amsmath}
\usepackage{threeparttable}
\usepackage{extarrows}
\usepackage{url,color}
\usepackage{graphicx,subfig}
\graphicspath{{./pics/}}
\usepackage{bm,color,url}
\usepackage{algorithm,algorithmic}
\usepackage{geometry}
\usepackage{setspace}
\usepackage[notref,notcite]{showkeys}

\allowdisplaybreaks[4]

\newtheorem{Theorem}{Theorem}[section]
\newtheorem{Lemma}[theorem]{Lemma}

\newtheorem{Definition}[theorem]{Definition}

\newtheorem{Remark}{Remark}[section]

\def\relu{{\mathrm{ReLU}}}

\def\relu{\mathrm{ReLU}}
\def\relu2{\mathrm{ReLU}^2}
\def\calD{\mathcal{D}}

\title{Convergence Analysis of the Deep Galerkin Method for Weak Solutions}
\author{
	Yuling Jiao\thanks{School of Mathematics and Statistics, and
		Hubei Key Laboratory of Computational Science, Wuhan University, Wuhan 430072, P.R. China. (yulingjiaomath@whu.edu.cn)}\quad\and
	Yanming Lai \thanks{School of Mathematics and Statistics,  Wuhan University, Wuhan 430072, P.R. China. (laiyanming@whu.edu.cn)}\quad\and
 Yang Wang \thanks{Department of Mathematics, The Hong Kong University of Science and Technology,
Clear Water Bay, Kowloon, Hong Kong (yangwang@ust.hk)}\quad \and
Haizhao Yang \thanks{Department of Mathematics, University of Maryland, College Park, MD 20742, USA (hzyang@umd.edu)} \quad
\and
Yunfei Yang  \thanks{Department of Mathematics, The Hong Kong University of Science and Technology,
Clear Water Bay, Kowloon, Hong Kong (yyangdc@connect.ust.hk)}
}

\begin{document}
	\maketitle
	
\begin{abstract}
	This paper analyzes the convergence rate of a deep Galerkin method for the weak solution (DGMW) of second-order elliptic partial differential equations on $\mathbb{R}^d$ with Dirichlet, Neumann, and Robin boundary conditions, respectively.  In DGMW, a deep neural network is applied to parametrize the PDE solution, and a second neural network is adopted to parametrize the test function in the traditional Galerkin formulation. By properly choosing the depth and width of these two networks in terms of the number of training samples $n$, it is shown that the convergence rate of DGMW is $\mathcal{O}(n^{-1/d})$,  which is the first convergence result for weak solutions.  The main idea of the proof is to divide the error of the DGMW  into an approximation error and a statistical error. We derive an upper bound on the approximation error in the $H^{1}$ norm and bound the statistical error via Rademacher complexity.
\end{abstract}

\section{Introduction}
\label{sec:introduction}

Deep learning \cite{deep2016Goodfellow} has achieved many breakthroughs in high-dimensional data analysis, e.g., in computer vision and natural language processing \cite{imagenet2017Krizhevsky,seq2014Sutskever}. Its outstanding performance has also motivated its application to solve high-dimensional PDEs, which is a challenging task for classical numerical methods, e.g., finite element methods \cite{the2012hughes} and finite difference methods \cite{numerical2013thomas}. The application of neural networks to solve PDEs dates back to the 1990s \cite{artificial1998Lagaris} for low-dimensional problems. In recent years, neural network-based PDE solvers were revisited for high-dimensional PDEs with tremendous successes and new development \cite{dgm15,raissi2019physics,dgm2018justin,the2017E,wan2020zang}. The key idea of these methods is to approximate the solutions of PDEs by neural networks and construct loss functions based on equations and their boundary conditions. \cite{raissi2019physics,dgm2018justin} use the squared residuals on the domain as the loss function and treat boundary conditions as penalty terms, which are called physics-informed neural networks (PINNs). Inspired by the Ritz method, \cite{the2017E} proposes the deep Ritz method (DRM) and uses variational forms of PDEs as loss functions. The idea of the Galerkin method has also been used in \cite{wan2020zang}, where, they propose a minimax training procedure via reformulating the problem of finding the weak solution of PDEs into minimizing an operator norm defined through a maximization problem induced by the weak formulation. Here we call the scheme inspired by the Galerkin method DGMW for short (In the original paper \cite{wan2020zang}, this method is called \textit{Weak Adversarial Network } method and called \textit{WAN} for short).

\subsection{Related works and our contributions}
Although there are great empirical achievements of deep learning methods for PDEs in recent
several years, a challenging and interesting  question is to provide a
rigorous error analysis such as the finite element method.
Several recent efforts have been devoted to making processes along this line.
 The error analysis of DRM has been studied in \cite{lu2021priori,xu2020finite,siegel2020approximation,duan2021convergence,jiao2021error,dondl2021uniform,lu2021machine,dondl2021uniform}. \cite{lu2021priori} concerns a priori generalization analysis of the deep Ritz method with two-layer neural networks, under the a priori assumption that the exact solutions of the PDEs lie in spectral Barron space. See also \cite{xu2020finite} for handling general equations with solutions living in spectral Barron space via two-layer $\mathrm{ReLU}^{k}$ networks. \cite{duan2021convergence,jiao2021error,lu2021machine} studied the  error analysis of the DRM in Sobolev spaces with deep networks.
\cite{shin2020convergence,mishra2020estimates,shin2020error,jiao2021convergence,lu2021machine} considered the  convergence and convergence rate of PINNs.

  Since the training loss of DGMW is in a minimax form and  there are two networks to train,
   it is much more challenging to provide a theoretical guarantee for DGMW than that of
DRM and PINNs.   As far as we know, there is no convergence result of  DGMW despite the excellent
  numerical performance shown in \cite{wan2020zang}.
In this paper, we give the first convergence rate analysis of  DGMW  to solve second-order elliptic equations with  Dirichlet, Neumann, and Robin boundary conditions, respectively, with deep neural networks in Sobolev spaces.  Our results show how to set the hyper-parameters of depth and width to achieve the desired convergence rate in terms of the number of training samples.
 The main contributions of this paper are summarized as follows.

\begin{itemize}
	\item We derive novel error decomposition results for DGMW,  which is of independent interest for minimax training with deep networks.
	\item We establish the first convergence rate of the DGMW  with  Drichilet, Neumann, and Robin boundary conditions.  $\forall \epsilon>0$, we prove that if we set the number of samples as $\mathcal{O}(\epsilon^{-d\log d})$ and    the depth, width and the bound of the weights in the two  networks   to be
	      \begin{equation*}
		      \calD\leq\mathcal{O}(\log d), \quad \mathcal{W}\leq \epsilon^{-d}, \quad B_{\theta}\leq\mathcal{O}(\epsilon^{\frac{-9d-8}{2}}),
	      \end{equation*}
	      then the $H^1$ norm error of DGMW in expectation is smaller than $\epsilon$.
\end{itemize}

\subsection{Organization}
The outline of the rest of this paper is as follows.  In Section \ref{sec:deep}, the error decomposition of the DGMW is given, while the details of approximation error and statistical error are presented in Section \ref{sec:app} and \ref{sec:sta}, respectively. We devote Section \ref{sec:rate} to the convergence rate of the DGMW.  Finally, we give a conclusion and extension in Section \ref{sec:conclusion}.

We end up this section with some notations used throughout this paper.
 Let $\mathcal{D}\in\mathbb{N}^+$. A function $\mathbf{f}: \mathbb{R}^{d} \rightarrow \mathbb{R}^{n_{\mathcal{D}}}$ implemented by a neural network is defined by
\begin{equation}\label{nn}
	\begin{array}{l}
		\mathbf{f}_{0}(\mathbf{x})=\mathbf{x},\\
		\mathbf{f}_{\ell}(\mathbf{x})=\mathbf{\rho}\left(A_{\ell} \mathbf{f}_{\ell-1}+\mathbf{b}_{\ell}\right)
		\quad \text { for } \ell=1, \ldots, \mathcal{D}-1, \\
		\mathbf{f}:=\mathbf{f}_{\mathcal{D}}(\mathbf{x})=A_{\mathcal{D}}\mathbf{f}_{\mathcal{D}-1}+\mathbf{b}_{\mathcal{D}},
	\end{array}
\end{equation}
where $A_{\ell}=\left(a_{ij}^{(\ell)}\right)\in\mathbb{R}^{n_{\ell}\times n_{\ell-1}}$ and $\mathbf{b}_{\ell}=\left(b_i^{(\ell)}\right)\in\mathbb{R}^{n_{\ell}}$. $\rho$ is called the activation function and acts componentwise. $\mathcal{D}$ is called the depth of the network and $\mathcal{W}:=\max\{n_{\ell}:\ell=1,\cdots,\mathcal{D}\}$ is called the width of the network. $\phi = \{A_{\ell},\mathbf{b}_{\ell}\}_{\ell}$ are called the weight parameters. For convenience, we denote $\mathfrak{n}_i$, $i=1,\cdots,\mathcal{D}$, as the number of nonzero weights on the first $i$ layers in the representation (\ref{nn}). Clearly $\mathfrak{n}_{\mathcal{D}}$ is the total number of nonzero weights. Sometimes we denote a function implemented by a neural network as $\mathbf{f}_{\rho}$ for short. We use the notation $\mathcal{N}_{\rho}\left(\mathcal{D}, \mathfrak{n}_{\mathcal{D}}, B_{\theta}\right)$ to refer to the collection of functions implemented by a $\rho-$neural network with depth $\mathcal{D}$, total number of nonzero weights $\mathfrak{n}_{\mathcal{D}}$ and each weight being bounded by $B_{\theta}$.

\section{Error Decomposition}\label{sec:deep}
We consider the following second-order divergence form in the elliptic equation:
\begin{equation} \label{second order elliptic equation}
		-\sum_{i,j=1}^{d}\partial_j(a_{ij}\partial_iu)+\sum_{i=1}^{d}b_i\partial_iu+cu=f  \quad\text { in } \Omega
\end{equation}
with three kinds of boundary conditions:
\begin{subequations}
\begin{align}
	u&=0\text { on } \partial \Omega\label{dirichlet}\\
	\sum_{i,j=1}^{d}a_{ij}\partial_iun_j&=g\text { on } \partial \Omega\label{neumann}\\
	\alpha u+\beta\sum_{i,j=1}^{d}a_{ij}\partial_iun_j&=g\text { on } \partial \Omega,\quad\alpha,\beta\in\mathbb{R},\beta\neq0\label{robin}
\end{align}
\end{subequations}
which are called the Drichilet, Neumann, and Robin boundary conditions, respectively. Note for Drichilet problem, we only consider the homogeneous boundary condition here since the inhomogeneous case can be turned into a homogeneous case by translation. We also remark that Neumann condition $(\ref{neumann})$ is covered by Robin condition $(\ref{robin})$. Hence in the following, we only consider Dirichlet problem and Robin problem.

We make the following assumption on the known terms in the equation:
\begin{enumerate}[({A}1)]
	\item $\quad\  f\in L^2(\Omega)$, $g\in L^{2}(\partial\Omega)$, $a_{ij}\in C(\bar{\Omega})$, $b_i,c\in L^{\infty}(\Omega)$, $c>0$
	\item $\quad\ $there exists $\lambda,\Lambda>0$ such that	$\lambda|\xi|^2\leq \sum_{i,j=1}^da_{ij}\xi_i\xi_j\leq\Lambda|\xi|^2 ,\quad\forall x\in\Omega,\xi\in\mathbb{R}^d$
	\item $\quad\ 4\lambda c>{d}\max_{1\leq i\leq d}\|b_i\|_{L^{\infty}(\Omega)}^2$
\end{enumerate}
In the following we abbreviate $C\left(\|f\|_{L^2(\Omega)},\|g\|_{L^{2}(\partial\Omega)},\|a_{ij}\|_{C(\bar{\Omega})},\|b_i\|_{L^{\infty}(\Omega)},\|c\|_{L^{\infty}(\Omega)},\lambda\right)$, constants depending on the known terms in equation, as $C(coe)$ for simplicity.

Under the above assumptions, a coercity result is easily acquired.
\begin{Lemma}\label{coercity}
	Let (A1)-(A3) holds. For any $u\in H^1(\Omega)$,
	\begin{equation*}
		\sum_{i,j=1}^{d}(a_{ij}\partial_i u,\partial_ju)+\sum_{i=1}^{d}(b_i\partial_iu,u)+(cu,u)\geq C(d,coe)\|u\|_{H^1(\Omega)}^2
	\end{equation*}
\end{Lemma}
\begin{proof}
	Applying H$\mathrm{\ddot{o}}$lder and Cauchy's inequality and choosing $\delta$ such that
	\begin{equation*}
		\frac{{d}\max_{1\leq i\leq d}\|b_i\|_{L^{\infty}(\Omega)}^2}{4c}<\delta<\lambda	
	\end{equation*}
	we have
	\begin{equation*}
		\begin{aligned}
			&\sum_{i,j=1}^{d}(a_{ij}\partial_i u,\partial_ju)+\sum_{i=1}^{d}(b_i\partial_iu,u)+(cu,u)\\
			&\geq \lambda|u|_{H^1(\Omega)}^2+c\|u\|_{L^2(\Omega)}^2
			-\sqrt{d}\max_{1\leq i\leq d}\|b_i\|_{L^{\infty}(\Omega)}\|u\|_{L^2(\Omega)}|u|_{H^1(\Omega)}\\
			&\geq(\lambda-\delta)|u|_{H^1(\Omega)}^2+\left(c-\frac{{d}\max_{1\leq i\leq d}\|b_i\|_{L^{\infty}(\Omega)}^2}{4\delta}\right)\|u\|_{L^2(\Omega)}^2
			\geq C(coe)\|u\|_{H^1(\Omega)}^2
		\end{aligned}	
	\end{equation*}
\end{proof}

The coercity ensures the existence and uniqueness of the weak solution of Dirichlet problem and Robin problem. Specifically, for problem $(\ref{second order elliptic equation})(\ref{dirichlet})$, the variational problems is: find $u\in H_0^1(\Omega)$ such that
\begin{equation} \label{variational dirichlet}
	\sum_{i,j=1}^{d}(a_{ij}\partial_iu,\partial_jv)+\sum_{i=1}^{d}(b_i\partial_iu,v)+(cu,v)=(f,v)  \quad\forall v\in H_{0}^1(\Omega)
\end{equation}

\begin{Lemma} \label{uD regularity}
	Let (A1)-(A3) holds. Let $u_D$ be the solution of problem $(\ref{variational dirichlet})$. Then $u_D\in H^2(\Omega)$.
\end{Lemma}
\begin{proof}
	See \cite{evans1998partial}.
\end{proof}

For problem $(\ref{second order elliptic equation})(\ref{robin})$, the variational problem is: find $u\in H^1(\Omega)$ such that
	\begin{equation} \label{variational robin}
	\sum_{i,j=1}^{d}(a_{ij}\partial_iu,\partial_jv)+\sum_{i=1}^{d}(b_i\partial_iu,v)+(cu,v)+\frac{\alpha}{\beta}(T_0u,T_0v)|_{\partial\Omega}=(f,v)+\frac{1}{\beta}(g,T_0v)|_{\partial\Omega},\quad\forall v\in H^1(\Omega)
	\end{equation}
	where $T_0$ is a zero-order trace operator.
	
\begin{Lemma} \label{uR regularity}
	Let (A1)-(A3) holds. Let $u_R$ be the solution of problem $(\ref{variational robin})$. Then $u_R\in H^2(\Omega)$ and $\|u_R\|_{H^2(\Omega)}\leq\frac{ C(coe)}{\beta}$ for any $\beta>0$.
\end{Lemma}
\begin{proof}
	See \cite{gilbarg2015elliptic}.
\end{proof}

Intuitively, when $\alpha=1$, $g=0$, and $\beta\to0$, we expect that the solution of the Robin problem converges to the solution of the Dirichlet problem. Hence we only need to consider the Robin problem since the Dirichlet problem can be handled through a limiting process. The next lemma verifies this assertion.
\begin{Lemma} \label{penalty convergence}
	Let (A1)-(A3) holds. Let $\alpha=1,g=0$. Let $u_D$ be the solution of problem $(\ref{variational dirichlet})$ and $u_R$ the solution of problem $(\ref{variational robin})$. There holds
	\begin{equation*}
		\|u_R-u_D\|_{H^1(\Omega)}\leq C(d,\Omega,coe)\beta^{1/2}
	\end{equation*}
\end{Lemma}
\begin{proof}
By the definition of $u_R$ and $u_D$, we have for any $v\in H^1(\Omega)$,
\begin{equation}\label{penalty1}
	\sum_{i,j=1}^{d}(a_{ij}\partial_i u_R,\partial_jv)+\sum_{i=1}^{d}(b_i\partial_iu_R,v)+(cu_R,v)+\frac{1}{\beta}(T_0u_R,T_0v)|_{\partial\Omega}=(f,v)
\end{equation}

\begin{equation}\label{penalty2}
	\sum_{i,j=1}^{d}(a_{ij}\partial_i u_D,\partial_jv)+\sum_{i=1}^{d}(b_i\partial_iu_D,v)+(cu_D,v)
	=(f,v)+\sum_{i,j=1}^{d}\int_{\partial\Omega}a_{ij}\partial_i u_D T_0vn_jds
\end{equation}
where $n_j$ is the $j$th component of $\mathbf{n}$,
the outward pointing unit normal vector along $\partial\Omega$. Subtracting $(\ref{penalty2})$ from $(\ref{penalty1})$ and choosing $v=u_R-u_D$, we have
\begin{align}
	&\sum_{i,j=1}^{d}(a_{ij}\partial_i (u_R-u_D),\partial_j(u_R-u_D))+\sum_{i=1}^{d}(b_i\partial_i(u_R-u_D),(u_R-u_D))\nonumber\\
	&+(c(u_R-u_D),(u_R-u_D))
	+\frac{1}{\beta}(T_0(u_R-u_D),T_0(u_R-u_D))|_{\partial\Omega}\nonumber\\
	&=\sum_{i,j=1}^{d}\int_{\partial\Omega}a_{ij}\partial_i u_D T_0(u_R-u_D)n_jds\label{penalty3}
\end{align}
where we use the fact that $T_0u_D=0$. For the term in the right hand side of $(\ref{penalty3})$, by H$\mathrm{\ddot{o}}$lder inequality and Cauchy's inequality, we have
\begin{align}
	&\sum_{i,j=1}^{d}\int_{\partial\Omega}a_{ij}\partial_i u_D T_0(u_R-u_D)n_jds\nonumber\\
	&\leq\max_{1\leq i,j\leq d}\|a_{ij}\|_{C(\bar{\Omega})}d^{3/2}|T_0u_D|_{H^1(\partial\Omega)}\|T_0(u_R-u_D)\|_{L^2(\partial\Omega)}\nonumber\\
	&\leq\frac{1}{4}\beta\left(\max_{1\leq i,j\leq d}\|a_{ij}\|_{C(\bar{\Omega})}\right)^2d^{3}|T_0u_D|_{H^1(\partial\Omega)}^2+\frac{1}{\beta}\|T_0(u_R-u_D)\|_{L^2(\partial\Omega)}^2\nonumber\\
	&\leq\frac{1}{4}\beta\left(\max_{1\leq i,j\leq d}\|a_{ij}\|_{C(\bar{\Omega})}\right)^2d^{3}C(\Omega)\|u_D\|_{H^2(\Omega)}^2+\frac{1}{\beta}\|T_0(u_R-u_D)\|_{L^2(\partial\Omega)}^2\label{penalty4}
\end{align}
where in the final step we apply the trace theorem
\begin{equation*}
	\|T_0v\|_{L^2(\partial\Omega)}\leq C(\Omega)\|v\|_{H^1(\Omega)}
\end{equation*}
See more details in \cite{adams2003sobolev}. Now combining Lemma $\ref{coercity}$, $(\ref{penalty3})$ and $(\ref{penalty4})$ yields the result.
\end{proof}

Define
\begin{equation*}
\mathcal{L}(u,v):=	\sum_{i,j=1}^{d}(a_{ij}\partial_iu,\partial_jv)+\sum_{i=1}^{d}(b_i\partial_iu,v)+(cu,v)+\frac{\alpha}{\beta}(T_0u,T_0v)|_{\partial\Omega}-(f,v)-\frac{1}{\beta}(g,T_0v)|_{\partial\Omega}
\end{equation*}
It is clear that if $u$ is the solution of problem $(\ref{variational robin})$, then
it solves the following optimization problem:
\begin{equation} \label{inf sup}
\inf_{u\in H^1(\Omega)}\sup_{\substack{v\in H^1(\Omega)\\ \|v\|_{H^1(\Omega)}\leq1}}\mathcal{L}(u,v)
\end{equation}
Note that $\mathcal{L}(u,v)$ can be equivalently written as
\begin{align*}
	\mathcal{L}(u,v)=&|\Omega|\mathbb{E}_{X\sim U(\Omega)}\left(\sum_{i,j=1}^{d}(a_{ij}\partial_iu\partial_jv)(X)+\sum_{i=1}^{d}(b_i\partial_iuv)(X)+(cuv)(X)-(fv)(X)\right)\\
	&+\frac{|\partial\Omega|}{\beta}\mathbb{E}_{Y\sim U(\partial\Omega)}\left(\frac{\alpha}{2}(T_0uT_0v)(Y)-(gT_0v)(Y)\right)
\end{align*}
where $U(\Omega)$ and $U(\partial\Omega)$ are uniform distribution on $\Omega$ and $\partial\Omega$, respectively. We then introduce a discrete version of $\mathcal{L}$ defined on $C^1(\Omega)\times C^1(\Omega)$:
\begin{align*}
	\widehat{\mathcal{L}}(u,v):=&\frac{|\Omega|}{N}\sum_{k=1}^N\left(\sum_{i,j=1}^{d}(a_{ij}\partial_iu\partial_jv)(X_k)+\sum_{i=1}^{d}(b_i\partial_iuv)(X_k)+(cuv)(X_k)-(fv)(X_k)\right)\\
	&+\frac{|\partial\Omega|}{\beta M}\sum_{k=1}^M\left(\frac{\alpha}{2}(T_0uT_0v)(Y_k)-(gT_0v)(Y_k)\right)
\end{align*}
where $\{X_k\}_{k=1}^{N}$ and $\{Y_k\}_{k=1}^{M}$ are i.i.d. random variables according to $U(\Omega)$ and $U(\partial\Omega)$ respectively. We now consider a minimax problem with respect to $\widehat{\mathcal{L}}$:
\begin{equation} \label{optimization}
\inf_{u\in\mathcal{P}}\sup_{v\in \mathcal{P}}\widehat{\mathcal{L}}(u,v)
\end{equation}
where $\mathcal{P}\subset C^1(\Omega)$ refers to the parameterized function class. Finally, we call a (random) solver $\mathcal{A}$, say SGD, to
minimize $\sup_{v\in \mathcal{P}}\widehat{\mathcal{L}}(\cdot,v)$ and denote the output of $\mathcal{A}$, say $u_{\phi_\mathcal{A}}$, as the final solution.

In order to study the difference between the weak solution of PDE $(\ref{second order elliptic equation})$($u_R$ and $u_D$) and the solution of empirical loss generated by a random solver ($u_{\phi_{\mathcal{A}}}$), we first define for any $u\in H^1(\Omega)$,
\begin{align*}
	&\mathcal{L}_0(u):=\sup_{\substack{v\in H^2(\Omega)\\ \|v\|_{H^2(\Omega)}\leq1}}\mathcal{L}(u,v)\\
	&\mathcal{L}_1(u):=\sup_{v\in \mathcal{P}}\mathcal{L}(u,v)\\
	&\mathcal{L}_2(u):=\sup_{v\in \mathcal{P}}\widehat{\mathcal{L}}(u,v)
\end{align*}
The following result decomposes the total error into three parts, enabling us to apply different methods to deal with different kinds of errors.
\begin{proposition} \label{error decomposition}
	Let (A1)-(A3) holds.
	Assume that $\mathcal{P}\subset C^1(\Omega)\bigcap H^2(\Omega)$ and $\|u\|_{H^1(\Omega)}\leq\mathcal{M}$ for all $u\in\mathcal{P}$.
	Let $u_R$ and $u_D$ be the solution of problem $(\ref{variational robin})$ and $(\ref{variational dirichlet})$, repsectively. Let $u_{\phi_{\mathcal{A}}}$ be the solution of problem $(\ref{optimization})$ generated by a random solver.
	
	(1)There holds
\begin{align*}
	\|u_{\phi\mathcal{A}}-u_R\|_{H^1(\Omega)}\leq C(d,\Omega,coe)\left(\mathcal{E}_{app}
	+\mathcal{E}_{sta}+\mathcal{E}_{opt}\right)
\end{align*}
with
\begin{align}
\mathcal{E}_{app}&:=\frac{\mathcal{M}}{\beta}\sup_{\substack{v_1\in H^2(\Omega)\\ \|v_1\|_{H^2(\Omega)}\leq1}}\inf_{v_2\in \mathcal{P}}\|v_1-v_2\|_{H^1(\Omega)}+\frac{\mathcal{M}}{\beta}\inf_{\bar{u}\in\mathcal{P}}\|\bar{u}-u_R\|_{H^1(\Omega)}\label{app}\\
\mathcal{E}_{sta}&:=2 \sup _{u \in \mathcal{P}}\left|\mathcal{L}_1(u)-{\mathcal{L}}_2(u)\right|\label{sta}\\
\mathcal{E}_{opt}&:={\mathcal{L}_2}\left(u_{\phi_{\mathcal{A}}}\right)-\inf_{u\in\mathcal{P}}{\mathcal{L}_2}\left(u\right)\label{opt}
\end{align}
	
	(2) Set $\alpha=1,g=0$. There holds
\begin{align*}
	\|u_{\phi\mathcal{A}}-u_R\|_{H^1(\Omega)}\leq C(d,\Omega,coe)\left(\mathcal{E}_{app}
	+\mathcal{E}_{sta}+\mathcal{E}_{opt}+\mathcal{E}_{pen}\right)
\end{align*}
where $\mathcal{E}_{app},\mathcal{E}_{sta},\mathcal{E}_{opt}$ are given by $(\ref{app}),(\ref{sta}),(\ref{opt})$ and
\begin{equation*}
\mathcal{E}_{pen}:=\|u_R-u_D\|_{H^1(\Omega)}
\end{equation*}
\end{proposition}
\begin{proof}
	We only prove (1) since (2) is a direct result of (1) and the triangle inequality.

Letting $\bar{u}$ be any element in $\mathcal{P}$, we have
	\begin{equation*}
\begin{split}
	&\mathcal{L}_0\left(u_{\phi_{\mathcal{A}}}\right)-\mathcal{L}_0\left(u_R\right) \\
	&=\mathcal{L}_0\left(u_{\phi_{\mathcal{A}}}\right)-\mathcal{L}_1\left(u_{\phi_{\mathcal{A}}}\right)+\mathcal{L}_1\left(u_{\phi_{\mathcal{A}}}\right)-{\mathcal{L}_2}\left(u_{\phi_{\mathcal{A}}}\right) +{\mathcal{L}_2}\left(u_{\phi_{\mathcal{A}}}\right)-\inf_{u\in\mathcal{P}}{\mathcal{L}_2}\left(u\right)\\
	&\quad+\inf_{u\in\mathcal{P}}{\mathcal{L}_2}\left(u\right)-{\mathcal{L}_2}\left(\bar{u}\right)
	+{\mathcal{L}_2}\left(\bar{u}\right)-\mathcal{L}_1\left(\bar{u}\right)+\mathcal{L}_1\left(\bar{u}\right)-\mathcal{L}_1\left(u_R\right) \\
	&\leq [\mathcal{L}_0\left(u_{\phi_{\mathcal{A}}}\right)-\mathcal{L}_1\left(u_{\phi_{\mathcal{A}}}\right)]+\left[\mathcal{L}_1\left(\bar{u}\right)-\mathcal{L}_1\left(u_R\right)\right]+2 \sup _{u \in \mathcal{P}}\left|\mathcal{L}_1(u)-{\mathcal{L}_2}(u)\right| +\left[{\mathcal{L}_2}\left(u_{\phi_{\mathcal{A}}}\right)-\inf_{u\in\mathcal{P}}{\mathcal{L}_2}\left(u\right)\right],
\end{split}
	\end{equation*}
	where we use the fact that $\mathcal{L}_0(u_R)=\mathcal{L}_1(u_R)=0$. Since $\bar{u}$ can be any element in $\mathcal{P}$, we take the infimum of $\bar{u}$ on both side of the above display,
	\begin{align}
		\mathcal{L}_0\left(u_{\phi_{\mathcal{A}}}\right)-\mathcal{L}_0\left(u_R\right)
		&\leq [\mathcal{L}_0\left(u_{\phi_{\mathcal{A}}}\right)-\mathcal{L}_1\left(u_{\phi_{\mathcal{A}}}\right)]+\inf_{\bar{u}\in\mathcal{P}}\left[\mathcal{L}_1\left(\bar{u}\right)-\mathcal{L}_1\left(u_R\right)\right]\nonumber\\
		&+2 \sup _{u \in \mathcal{P}}\left|\mathcal{L}_1(u)-{\mathcal{L}}_2(u)\right| +\left[{\mathcal{L}_2}\left(u_{\phi_{\mathcal{A}}}\right)-\inf_{u\in\mathcal{P}}{\mathcal{L}_2}\left(u\right)\right]\label{errdec1}
	\end{align}
Now for the term on the left hand side of $(\ref{errdec1})$, by Lemma $\ref{coercity}$ we have
\begin{align}
	&\mathcal{L}_0\left(u_{\phi_{\mathcal{A}}}\right)-\mathcal{L}_0\left(u_R\right)
	=\sup_{\substack{v\in H^2(\Omega)\\ \|v\|_{H^2(\Omega)}\leq1}}\left[\mathcal{L}(u_{\phi_\mathcal{A}},v)-\mathcal{L}(u_{R},v)\right]\nonumber\\
	&=\sup_{\substack{v\in H^2(\Omega)\\ \|v\|_{H^2(\Omega)}\leq1}}\left[\sum_{i,j=1}^{d}(a_{ij}\partial_i (u_{\phi_{\mathcal{A}}}-u_R),\partial_jv)+\sum_{i=1}^{d}(b_i\partial_i(u_{\phi_{\mathcal{A}}}-u_R),v)\right.\nonumber\\
	&\qquad\qquad\qquad\ +\left.(c(u_{\phi_{\mathcal{A}}}-u_R),v)+\frac{\alpha}{\beta}(T_0(u_{\phi_{\mathcal{A}}}-u_R),T_0v)|_{\partial\Omega}\right]\nonumber\\
	&\geq\sum_{i,j=1}^{d}\left(a_{ij}\partial_i (u_{\phi_{\mathcal{A}}}-u_R),\frac{\partial_j(u_{\phi_{\mathcal{A}}}-u_R)}{\|u_{\phi_{\mathcal{A}}}-u_R\|_{H^1(\Omega)}}\right)+\sum_{i=1}^{d}\left(b_i\partial_i(u_{\phi_{\mathcal{A}}}-u_R),\frac{u_{\phi_{\mathcal{A}}}-u_R}{\|u_{\phi_{\mathcal{A}}}-u_R\|_{H^1(\Omega)}}\right)\nonumber\\
	&\quad+\left(c(u_{\phi_{\mathcal{A}}}-u_R),\frac{u_{\phi_{\mathcal{A}}}-u_R}{\|u_{\phi_{\mathcal{A}}}-u_R\|_{H^1(\Omega)}}\right)+\frac{\alpha}{\beta}\left(T_0(u_{\phi_{\mathcal{A}}}-u_R),\frac{T_0(u_{\phi_{\mathcal{A}}}-u_R)}{\|u_{\phi_{\mathcal{A}}}-u_R\|_{H^1(\Omega)}}\right)\left|_{\partial\Omega}\right.\nonumber\\
	&\geq C(d,coe)\|u_{\phi_{\mathcal{A}}}-u_R\|_{H^1(\Omega)},\label{errdec2}
\end{align}
where the first step is due to the fact that $\mathcal{L}_0(u_R)=0$. For the first term on the right-hand side of $(\ref{errdec1})$,
\begin{align}
	&\mathcal{L}_0\left(u_{\phi_{\mathcal{A}}}\right)-\mathcal{L}_1\left(u_{\phi_{\mathcal{A}}}\right)
	=\sup_{\substack{v\in H^2(\Omega)\\ \|v\|_{H^2(\Omega)}\leq1}}\mathcal{L}(u_{\phi_{\mathcal{A}}},v)-\sup_{v\in \mathcal{P}}\mathcal{L}(u_{\phi_{\mathcal{A}}},v)\nonumber\\
	&=\sup_{\substack{v_1\in H^2(\Omega)\\ \|v_1\|_{H^2(\Omega)}\leq1}}\inf_{v_2\in \mathcal{P}}\mathcal{L}(u_{\phi_{\mathcal{A}}},v_1-v_2)\nonumber\\
	&\leq\sup_{\substack{v_1\in H^2(\Omega)\\ \|v_1\|_{H^2(\Omega)}\leq1}}\inf_{v_2\in \mathcal{P}}\frac{1}{\beta}C(d,\Omega,coe)\|u_{\phi_{\mathcal{A}}}\|_{H^1(\Omega)}\|v_1-v_2\|_{H^1(\Omega)}+\frac{1}{\beta}C(d,\Omega,coe)\|v_1-v_2\|_{H^1(\Omega)}\nonumber\\
	&\leq\frac{\mathcal{M}}{\beta}C(d,\Omega,coe)\sup_{\substack{v_1\in H^2(\Omega)\\ \|v_1\|_{H^2(\Omega)}\leq1}}\inf_{v_2\in \mathcal{P}}\|v_1-v_2\|_{H^1(\Omega)}\label{errdec3}
\end{align}
For the second term on the right hand side of $(\ref{errdec1})$,
\begin{align}
	\mathcal{L}_1(\bar{u})-\mathcal{L}_1(u_R)&=\sup_{v\in \mathcal{P}}
	\left[\mathcal{L}(\bar{u},v)-\mathcal{L}(u_R,v)\right]\leq\frac{\mathcal{M}}{\beta}C(d,\Omega,coe)\|\bar{u}-u_R\|_{H^1(\Omega)}\label{errdec4}
\end{align}
Combining $(\ref{errdec1})-(\ref{errdec4})$ yields the result.
\end{proof}

\section{Approximation Error}\label{sec:app}
In this section, we study the approximation error $\mathcal{E}_{app}$ defined in $(\ref{app})$. Clearly, we first need a neural network approximation result in Sobolev spaces. In this field, \cite{guhring2021approximation} is a comprehensive study concerning a variety of activation functions, including $\mathrm{ReLU}$, sigmoidal type functions, etc. The key idea in \cite{guhring2021approximation} to build the upper bound in Sobolev spaces is to construct an approximate partition of unity.

Denote $\mathcal{F}_{s,p,d}:=\left\{f\in W^{s,p}\left([0,1]^d\right):\|f\|_{W^{s,p}\left([0,1]^d\right)}\leq1\right\}$.

\begin{Theorem}[Proposition 4.8, \cite{guhring2021approximation}]\label{general app}
	Let $p\geq1$, $s,k,d\in\mathbb{N}^+$, $s\geq k+1$,$\bar{k}\geq k$. Let $\rho$ be $\max\{0,x\}^{\bar{k}}$($\mathrm{ReLU}^{\bar{k}}$), $\frac{1}{1+e^{-x}}$(logistic function) or $\frac{e^x-e^{-x}}{e^x+e^{-x}}$(tanh function). For any $\epsilon>0$ and $f\in\mathcal{F}_{s,p,d}$, there exists a neural network $f_{\rho}$ with depth $C\log (d+s)$ such that
	\begin{equation*}
		\|f-f_{\rho}\|_{W^{k,p}\left([0,1]^d\right)}\leq\epsilon.
	\end{equation*}
(1) If $\rho=\max\{0,x\}^{\bar{k}}$, then the number of non-zero weights of $f_\rho$ is bounded by $C(d,s,p,k)\epsilon^{-d /(s-k)}$. (2) If $\rho=\frac{1}{1+e^{-x}}$ or  $\frac{e^x-e^{-x}}{e^x+e^{-x}}$, then the number of non-zero weights of $f_\rho$ is bounded by $C(d,s,p,k)\epsilon^{-d /(s-k-\mu k)}$. Moreover, in case(2) the value of weights is bounded in absolute value by
	\begin{equation*}
		C(d,s,p,k)\epsilon^{-2-\frac{2(d / p+d+k+\mu k)+d / p+d}{s-k-\mu k}}
	\end{equation*}
where $\mu$ is an arbitrarily small positive number.
\end{Theorem}

\begin{Remark}
	\textnormal{
		The bounds in the theorem can be found in the proof of \cite[Proposition 4.8]{guhring2021approximation}, except that they did not explicitly give the bound on the depth. In their proof, they partition $[0,1]^d$ into small patches, approximate $f$ by a sum of localized polynomial $\sum_m \phi_m p_m$, and approximately implement $\sum_m \phi_m p_m$ by a neural network, where the bump functions $\{\phi_m\}$ form an approximately partition of unity and $p_m = \sum_{|\alpha|<s}c_{f,m,\alpha}x^\alpha$ are the averaged Taylor polynomials. As shown in \cite{guhring2021approximation}, $\phi_m$ can be approximated by the products of the $d$-dimensional output of a neural network with constant layers. And the identity map $I(x)=x$ and the product function $\times(a,b) = ab$ can also be approximated by neural networks with constant layers. In order to approximate $\phi_m x^\alpha$, we need to implement $d+s-1$ products. Hence, the required depth can be bounded by $C\log(d+s)$.
	}
\end{Remark}

Since the region $[0,1]^d$ is larger than the region $\Omega$ we consider(recall we assume without loss of generality that $\Omega\subset[0,1]^d$ at the beginning), we need the following extension result.
\begin{Lemma} \label{extension}
	Let $k\in\mathbb{N}^+$, $1\leq p<\infty$. There exists a linear operator $E$ from $W^{k,p}(\Omega)$ to $W_0^{k,p}\left([0,1]^d\right)$ and $Eu=u$ in $\Omega$.
\end{Lemma}
\begin{proof}
	See Theorem 7.25 in \cite{gilbarg2015elliptic}.
\end{proof}

From Lemmas $\ref{uR regularity}$ and $(\ref{app})$, we know that we need to approximate functions in $H^1(\Omega)$ and our target functions lie in $H^2(\Omega)$. We immediately obtain the result we desire from Theorem $\ref{general app}$ and Lemma $\ref{extension}$.

\begin{Theorem} \label{app error}
Let $\rho$ be $\frac{1}{1+e^{-x}}$(logistic function) or $\frac{e^x-e^{-x}}{e^x+e^{-x}}$(tanh function). For any sufficiently small $\epsilon>0$, set the parameterized function class
\begin{align*}
\mathcal{P}:=\mathcal{N}_{\rho}\left(C\log(d+1),C(d,coe)(\beta^2\epsilon)^{\frac{-d}{1-\mu}},C(d,coe) (\beta^2\epsilon)^{\frac{-9d-8}{2-2\mu}}\right)\bigcap B_{H^1(\Omega)}(0,2)
\end{align*}
where $B_{H^1(\Omega)}(0,2):=\{f\in H^1(\Omega):\|f\|_{H^1(\Omega)}\leq2\}$, then $\mathcal{E}_{app}\leq\epsilon$ with $\mathcal{E}_{app}$ defined by $(\ref{app})$.
\end{Theorem}
\begin{proof}
Set $k=1$, $s=2$, $p=2$ in Theorem $\ref{general app}$ and use the fact $\|f-f_\rho\|_{H^2(\Omega)}\leq \|Ef-f_\rho\|_{H^2\left([0,1]^d\right)}$ with $E$ being the  extension operator in Lemma $\ref{extension}$, we conclude that for any $0<\delta\leq1$ and $f\in H^2(\Omega)$ with $\|f\|_{H^2(\Omega)}\leq1$, there exists a neural network $f_{\rho}$ with depth $C\log(d+1)$ and the number of weights $C(d)\delta^{-d/(1-\mu)}$ such that
	\begin{equation}\label{app error1}
		\|f-f_{\rho}\|_{H^1\left(\Omega\right)}\leq\delta
	\end{equation}
and the value of weights are bounded by $C(d) \delta^{-(9d+8)/(2-2\mu)}$, where $\mu$ is an arbitrarily small positive number. Denote
\begin{align*}
\mathcal{P}_\delta^0:=\{f_{\rho}:f\in H^2(\Omega),\quad\|f\|_{H^2(\Omega)}\leq1\}
\end{align*}
Clearly
\begin{align*}
\mathcal{P}_\delta^0\subset\mathcal{P}_\delta^1:=\mathcal{N}_{\rho}\left(C\log(d+1),C(d)\delta^{-d/(1-\mu)},C(d) \delta^{-(9d+8)/(2-2\mu)}\right)
\end{align*}
In addition, for any $f_\rho\in\mathcal{P}_\delta^0$,
\begin{align*}
\|f_\rho\|_{H^1(\Omega)}\leq\|f_\rho-f\|_{H^1(\Omega)}+\|f\|_{H^1(\Omega)}\leq\delta+1\leq2
\end{align*}
Therefore
\begin{align*}
\mathcal{P}_\delta^0\subset\mathcal{P}_\delta^1\bigcap B_{H^1(\Omega)}(0,2)
\end{align*}
with $B_{H^1(\Omega)}(0,2):=\{f\in H^1(\Omega):\|f\|_{H^1(\Omega)}\leq2\}$.

Now we set the parameterized function class
\begin{align*}
\mathcal{P}=\mathcal{P}_\delta^B:=\mathcal{P}_\delta^1\bigcap B_{H^1(\Omega)}(0,2)
\end{align*}
and estimate the approximation error $\mathcal{E}_{app}$ defined by $(\ref{app})$. We first normalize the second term in $(\ref{app})$.
\begin{align*}
\inf_{\bar{u}\in\mathcal{P}_\delta^B}\|\bar{u}-u_R\|_{H^1(\Omega)}
&=\|u_R\|_{H^1(\Omega)}\inf_{\bar{u}\in\mathcal{P}_\delta^B}\left\|\frac{\bar{u}}{\|u_R\|_{H^1(\Omega)}}-\frac{u_R}{\|u_R\|_{H^1(\Omega)}}\right\|_{H^1(\Omega)}\\
&=\|u_R\|_{H^1(\Omega)}\inf_{\bar{u}\in\mathcal{P}_\delta^B}\left\|\bar{u}-\frac{u_R}{\|u_R\|_{H^1(\Omega)}}\right\|_{H^1(\Omega)}\\
&\leq\frac{ C(coe)}{\beta}\inf_{\bar{u}\in\mathcal{P}_\delta^B}\left\|\bar{u}-\frac{u_R}{\|u_R\|_{H^1(\Omega)}}\right\|_{H^1(\Omega)}
\end{align*}
where in the third step we apply Lemma $\ref{uR regularity}$. Hence
\begin{align} \label{app error2}
\mathcal{E}_{app}\leq\frac{2}{\beta}\sup_{\substack{v_1\in H^2(\Omega)\\ \|v_1\|_{H^2(\Omega)}\leq1}}\inf_{v_2\in \mathcal{P}_\delta^B}\|v_1-v_2\|_{H^1(\Omega)}+\frac{2C(coe)}{\beta^2}\inf_{\bar{u}\in\mathcal{P}_\delta^B}\left\|\bar{u}-\frac{u_R}{\|u_R\|_{H^1(\Omega)}}\right\|_{H^1(\Omega)}
\end{align}
Setting $\delta=C(coe)\beta^2\epsilon$ and combining $(\ref{app error1})$ and $(\ref{app error2})$ yields the result.
\end{proof}

\section{Statistical Error} \label{sec:sta}
In this section, we study the statistical error $\mathcal{E}_{sta}$ defined by $(\ref{sta})$.
\begin{Lemma}\label{triangle inequality}
For the statistical error $\mathcal{E}_{sta}$, there holds
\begin{align*}
\mathcal{E}_{sta}\leq\sum_{k=1}^6I_k
\end{align*}
with
\begin{align*}
I_1&:=2|\Omega|\sup _{u,v \in \mathcal{P}}\left|\mathbb{E}_{X\sim U(\Omega)}\sum_{i,j=1}^{d}(a_{ij}\partial_iu\partial_jv)(X)-\frac{1}{N}\sum_{k=1}^N\sum_{i,j=1}^{d}(a_{ij}\partial_iu\partial_jv)(X_k)\right|\\
I_2&:=2|\Omega|\sup _{u,v \in \mathcal{P}}\left|\mathbb{E}_{X\sim U(\Omega)}\sum_{i=1}^{d}(b_i\partial_iuv)(X)-\frac{1}{N}\sum_{k=1}^N\sum_{i=1}^{d}(b_i\partial_iuv)(X_k)\right|\\
I_3&:=2|\Omega|\sup _{u,v \in \mathcal{P}}\left|\mathbb{E}_{X\sim U(\Omega)}(cuv)(X)-\frac{1}{N}\sum_{k=1}^N(cuv)(X_k)\right|\\
I_4&:=2|\Omega|\sup _{u,v \in \mathcal{P}}\left|\mathbb{E}_{X\sim U(\Omega)}(fv)(X)-\frac{1}{N}\sum_{k=1}^N(fv)(X_k)\right|\\
I_5&:=2\frac{|\partial\Omega|}{\beta}\sup _{u,v \in \mathcal{P}}\left|\mathbb{E}_{Y\sim U(\partial\Omega)}\frac{\alpha}{2}(T_0uT_0v)(Y)-\frac{1}{M}\sum_{k=1}^M\frac{\alpha}{2}(T_0uT_0v)(Y_k)\right|\\
I_6&:=2\frac{|\partial\Omega|}{\beta}\sup _{u,v \in \mathcal{P}}\left|\mathbb{E}_{Y\sim U(\partial\Omega)}(gT_0v)(Y)-\frac{1}{M}\sum_{k=1}^M(gT_0v)(Y_k)\right|\\
	\end{align*}
\end{Lemma}
\begin{proof}
We have
\begin{align*}
\mathcal{E}_{sta}&=2 \sup _{u \in \mathcal{P}}\left|\mathcal{L}_1(u)-{\mathcal{L}}_2(u)\right|=
2\sup _{u \in \mathcal{P}}\left|\sup_{v\in \mathcal{P}}\mathcal{L}(u,v)-\sup_{v\in \mathcal{P}}\widehat{\mathcal{L}}(u,v)\right|\\
&\leq2 \sup _{u \in \mathcal{P}}\sup_{v\in \mathcal{P}}\left|\mathcal{L}(u,v)-\widehat{\mathcal{L}}(u,v)\right|\leq\sum_{k=1}^6I_k
\end{align*}
where the third step is due to the fact that
\begin{align*}
\sup_{v\in \mathcal{P}}\mathcal{L}(u,v)-\sup_{v\in \mathcal{P}}\widehat{\mathcal{L}}(u,v)&\leq
\sup_{v\in \mathcal{P}}\left[\mathcal{L}(u,v)-\widehat{\mathcal{L}}(u,v)\right]\leq\sup_{v\in \mathcal{P}}\left|\mathcal{L}(u,v)-\widehat{\mathcal{L}}(u,v)\right|\\
\sup_{v\in \mathcal{P}}\widehat{\mathcal{L}}(u,v)-\sup_{v\in \mathcal{P}}\mathcal{L}(u,v)&\leq
\sup_{v\in \mathcal{P}}\left[\widehat{\mathcal{L}}(u,v)-\mathcal{L}(u,v)\right]\leq\sup_{v\in \mathcal{P}}\left|\mathcal{L}(u,v)-\widehat{\mathcal{L}}(u,v)\right|
\end{align*}
\end{proof}



By the technique of symmetrization, we can bound the difference between continuous loss and empirical loss(i.e., $I_1,\cdots,I_6$) by Rademacher complexity. We first introduce Rademacher complexity.
\begin{definition}
	The Rademacher complexity of a set $A \subseteq \mathbb{R}^N$ is defined as
	\begin{equation*}
		\mathfrak{R}_N(A) = \mathbb{E}_{\{\sigma_i\}_{k=1}^N}\left[\sup_{a\in A}\frac{1}{N}\sum_{k=1}^N \sigma_k a_k\right],
	\end{equation*}
	where,   $\{\sigma_k\}_{k=1}^N$ are $N$ i.i.d  Rademacher variables with $\mathbb{P}(\sigma_k = 1) = \mathbb{P}(\sigma_k = -1) = \frac{1}{2}.$
	The Rademacher complexity of  function class $\mathcal{F}$ associate with random sample $\{X_k\}_{k=1}^{N}$ is defined as
	\begin{equation*}
		\mathfrak{R}_N(\mathcal{F}) = \mathbb{E}_{\{X_k,\sigma_k\}_{k=1}^{N}}\left[\sup_{u\in \mathcal{F}}\frac{1}{N}\sum_{k=1}^N \sigma_k u(X_k)\right].
	\end{equation*}
\end{definition}

\begin{Lemma} \label{symmetrization}
There holds
\begin{align*}
&\mathbb{E}_{\{{X_k}\}_{k=1}^{N}}I_i\leq4|\Omega|\mathfrak{R}_N(\mathcal{F}_i),\quad i=1,\cdots,4\\
&\mathbb{E}_{\{{Y_k}\}_{k=1}^{M}}I_i\leq\frac{4|\partial\Omega|}{\beta}\mathfrak{R}_M(\mathcal{F}_i),\quad i=5,6
\end{align*}
with
\begin{align*}
&\mathcal{F}_1:=\left\{\sum_{i,j=1}^{d}a_{ij}\partial_iu\partial_jv:u,v\in\mathcal{P}\right\},
&&\mathcal{F}_2:=\left\{\sum_{i=1}^{d}b_i\partial_iuv:u,v\in\mathcal{P}\right\}\\
&\mathcal{F}_3:=\left\{cuv:u,v\in\mathcal{P}\right\},
&&\mathcal{F}_4:=\left\{fv:u,v\in\mathcal{P}\right\}\\
&\mathcal{F}_5:=\left\{\frac{\alpha}{2}T_0uT_0v:u,v\in\mathcal{P}\right\},
&&\mathcal{F}_6:=\left\{gT_0v:u,v\in\mathcal{P}\right\}
\end{align*}
\end{Lemma}
\begin{proof}
We only present the proof with respect to $I_3$ since other inequalities can be shown similarly. We take $\{\widetilde{X_k}\}_{k=1}^{N}$ as an independent copy of $\{{X_k}\}_{k=1}^{N}$, then
\begin{align*}
I_3&=2|\Omega|\sup_{u,v\in\mathcal{P}}\left|\mathbb{E}_{X\sim U(\Omega)}(cuv)(X)-\frac{1}{N}\sum_{k=1}^{N}(cuv)(X_k)\right|\\	
&\leq\frac{2|\Omega|}{N}\sup_{u,v\in\mathcal{P}}\left|\mathbb{E}_{\{\widetilde{X_k}\}_{k=1}^{N}}\sum_{k=1}^{N}\left[(cuv)(\widetilde{X_k})-(cuv)(X_k)\right]\right|\\
&\leq\frac{2|\Omega|}{N}\mathbb{E}_{\{\widetilde{X_k}\}_{k=1}^{N}}\sup_{u,v\in\mathcal{P}}\left|\sum_{k=1}^{N}\left[(cuv)(\widetilde{X_k})-(cuv)(X_k)\right]\right|
\end{align*}
Hence
\begin{align*}
\mathbb{E}_{\{{X_k}\}_{k=1}^{N}}I_3&\leq\frac{2|\Omega|}{N}\mathbb{E}_{\{X_k,\widetilde{X_k}\}_{k=1}^{N}}\sup_{u,v\in\mathcal{P}}\left|\sum_{k=1}^{N}\left[(cuv)(\widetilde{X_k})-(cuv)(X_k)\right]\right|\\
&=\frac{2|\Omega|}{N}\mathbb{E}_{\{X_k,\widetilde{X_k},\sigma_k\}_{k=1}^{N}}\sup_{u,v\in\mathcal{P}}\left|\sum_{k=1}^{N}\sigma_k\left[(cuv)(\widetilde{X_k})-(cuv)(X_k)\right]\right|\\
&=\frac{2|\Omega|}{N}\mathbb{E}_{\{X_k,\widetilde{X_k},\sigma_k\}_{k=1}^{N}}\sup_{u,v\in\mathcal{P}}\\
&\quad\ \max\left\{\sum_{k=1}^{N}\sigma_k\left[(cuv)(\widetilde{X_k})-(cuv)(X_k)\right],\sum_{k=1}^{N}\sigma_k\left[(cuv)({X_k})-(cuv)(\widetilde{X_k})\right]\right\}
\end{align*}
where the second step is due to the fact that the insertion of Rademacher variables doesn't change the distribution. We note that
\begin{align*}
&\mathbb{E}_{\{X_k,\widetilde{X_k},\sigma_k\}_{k=1}^{N}}\sup_{u,v\in\mathcal{P}}\sum_{k=1}^{N}\sigma_k\left[(cuv)(\widetilde{X_k})-(cuv)(X_k)\right]\\
&\leq\mathbb{E}_{\{X_k,\widetilde{X_k},\sigma_k\}_{k=1}^{N}}\sup_{u,v\in\mathcal{P}}\sum_{k=1}^{N}\sigma_k(cuv)(\widetilde{X_k})+\mathbb{E}_{\{X_k,\widetilde{X_k},\sigma_k\}_{k=1}^{N}}\sup_{u,v\in\mathcal{P}}\sum_{k=1}^{N}-\sigma_k(cuv)(X_k)\\
&=2\mathbb{E}_{\{X_k,\sigma_k\}_{k=1}^{N}}\sup_{u,v\in\mathcal{P}}\sum_{k=1}^{N}\sigma_k(cuv)(X_k)
\end{align*}
Similarly,
\begin{align*}
\mathbb{E}_{\{X_k,\widetilde{X_k},\sigma_k\}_{k=1}^{N}}\sup_{u,v\in\mathcal{P}}\sum_{k=1}^{N}\sigma_k\left[(cuv)({X_k})-(cuv)(\widetilde{X_k})\right]\leq2\mathbb{E}_{\{X_k,\sigma_k\}_{k=1}^{N}}\sup_{u,v\in\mathcal{P}}\sum_{k=1}^{N}\sigma_k(cuv)(X_k)
\end{align*}
Therefore
\begin{align*}
\mathbb{E}_{\{{X_k}\}_{k=1}^{N}}I_3\leq4|\Omega|\mathfrak{R}_N(\mathcal{F}_3)
\end{align*}

\end{proof}

In order to bound Rademacher complexities, we need the concept of covering numbers.
\begin{Definition}
An $\epsilon$-cover of a set $T$ in a metric space $(S, \tau)$
is a subset $T_c\subset S$ such  that for each $t\in T$, there exists a $t_c\in T_c$ such that $\tau(t, t_c) \leq\epsilon$. The $\epsilon$-covering number of $T$, denoted as $\mathcal{C}(\epsilon, T,\tau)$ is defined to be the minimum cardinality among all $\epsilon$-cover of $T$ with respect to the metric $\tau$.
\end{Definition}

In Euclidean space, we can establish an upper bound of the covering number for a bounded set easily.
\begin{Lemma} \label{covering number Euclidean space}
	Suppose that $T\subset\mathbb{R}^d$ and $\|t\|_2\leq B$ for $t\in T$, then
	\begin{equation*}
		\mathcal{C}(\epsilon,T,\|\cdot\|_2)\leq\left(\frac{2B\sqrt{d}}{\epsilon}\right)^d.
	\end{equation*}
\end{Lemma}
\begin{proof}
	Let $m=\left\lfloor\frac{2B\sqrt{d}}{\epsilon}\right\rfloor$ and define
	\begin{equation*}
		T_c=\left\{-B+\frac{\epsilon}{\sqrt{d}},-B+\frac{2\epsilon}{\sqrt{d}},\cdots,-B+\frac{m\epsilon}{\sqrt{d}}\right\}^d,
	\end{equation*}
	then for $t\in T$, there exists $t_c\in T_c$ such that
	\begin{equation*}
		\|t-t_c\|_2\leq\sqrt{\sum_{i=1}^{d}\left(\frac{\epsilon}{\sqrt{d}}\right)^2}=\epsilon.
	\end{equation*}
	Hence
	\begin{equation*}
		\mathcal{C}(\epsilon,T,\|\cdot\|_2)\leq|T_c|=m^d\leq\left(\frac{2B\sqrt{d}}{\epsilon}\right)^d.
	\end{equation*}
\end{proof}

A Lipschitz parameterization allows us to translate a cover of the function space into a cover of the parameter space. Such a property plays an essential role in our analysis of statistical error.
\begin{Lemma} \label{covering number Lipshcitz parameterization}
	Let $\mathcal{F}$ be a parameterized class of functions: $\mathcal{F} = \{f(x; \theta) : \theta\in\Theta \}$. Let $\|\cdot\|_{\Theta}$ be a norm on $\Theta$ and let $\|\cdot\|_{\mathcal{F}}$ be a norm on $\mathcal{F}$. Suppose that the mapping $\theta \mapsto f(x;\theta)$ is L-Lipschitz, that is,
	\begin{equation*}
		\left\|f(x;\theta)-f\left(x;\widetilde{\theta}\right)\right\|_{\mathcal{F}} \leq L\left\|\theta-\widetilde{\theta}\right\|_{\Theta},
	\end{equation*}
	then for any $\epsilon>0$, $\mathcal{C}\left(\epsilon, \mathcal{F},\|\cdot\|_{\mathcal{F}}\right) \leq \mathcal{C}\left(\epsilon / L, \Theta,\|\cdot\|_{\Theta}\right)$.
\end{Lemma}
\begin{proof}
	Suppose that $\mathcal{C}\left(\epsilon / L, \Theta,\|\cdot\|_{\Theta}\right)=n$ and $\{\theta_i\}_{i=1}^{n}$ is an $\epsilon / L$-cover of $\Theta$. Then for any $\theta\in\Theta$, there exists $1\leq i\leq n$ such that
	\begin{equation*}
		\left\|f(x;\theta)-f\left(x;{\theta}_i\right)\right\|_{\mathcal{F}} \leq L\left\|\theta-{\theta}_i\right\|_{\Theta}\leq\epsilon.
	\end{equation*}
	Hence $\{f(x;\theta_i)\}_{i=1}^{n}$ is an $\epsilon$-cover of $\mathcal{F}$, implying that $\mathcal{C}\left(\epsilon, \mathcal{F},\|\cdot\|_{\mathcal{F}}\right) \leq n$.
\end{proof}

To find the relation between Rademacher complexity and covering number, we first need the Massart¡¯s finite class lemma stated below.
\begin{Lemma}\label{Mfinite}
	For any finite set $A \subset \mathbb{R}^N$ with diameter $D = \sup_{a\in A}\|a\|_2$,
	\begin{equation*}
	\mathfrak{R}_N(A) \leq \frac{D}{N}\sqrt{2\log |A|}.
	\end{equation*}
\end{Lemma}
\begin{proof}
See, for example, \cite[Lemma 26.8]{shalev2014understanding}.
\end{proof}

\begin{Lemma} \label{chaining}
Let $\mathcal{F}$ be a function class and $\|f\|_{\infty}\leq B$ for any $f\in\mathcal{F}$, we have
\begin{equation*}
\mathfrak{R}_N(\mathcal{F})\leq\inf_{0<\delta<B/2}\left(4\delta+\frac{12}{\sqrt{N}}\int_{\delta}^{B/2}\sqrt{\log\mathcal{C}(\epsilon,\mathcal{F},\|\cdot\|_{\infty})}d\epsilon\right).
\end{equation*}
\end{Lemma}
\begin{proof}
We apply the chaining method. Set $\epsilon_k=2^{-k+1}B$. We denote by $\mathcal{F}_k$ such that $\mathcal{F}_k$ is an $\epsilon_k$-cover of $\mathcal{F}$ and $\left|\mathcal{F}_k\right|=\mathcal{C}(\epsilon_k,\mathcal{F},\|\cdot\|_{\infty})$. Hence for any $u\in\mathcal{F}$, there exists $u_k\in\mathcal{F}_k$ such that $\|u-u_k\|_\infty\leq\epsilon_k$. Let $K$ be a positive integer determined later. We have
\begin{align*}
	&\mathfrak{R}_N(\mathcal{F}) = \mathbb{E}_{\{\sigma_i,X_i\}_{i=1}^N} \left[\sup _{u \in \mathcal{F}} \frac{1}{N}\sum_{i=1}^{N} \sigma_{i} u\left(X_{i}\right) \right]\\\
	&=\mathbb{E}_{\{\sigma_i,X_i\}_{i=1}^N} \left[\frac{1}{N} \sup _{u \in \mathcal{F}} \sum_{i=1}^{N} \sigma_{i}\left(u\left(X_{i}\right)-u_K\left(X_{i}\right)\right)+\sum_{j=1}^{K-1} \sum_{i=1}^{N} \sigma_{i}\left(u_{j+1}\left(X_{i}\right)-u_j\left(X_{i}\right)\right)+\sum_{i=1}^{N} \sigma_{i} u_1\left(X_{i}\right)\right] \\
	&\leq \mathbb{E}_{\{\sigma_i,X_i\}_{i=1}^N} \left[\sup _{u \in \mathcal{F}} \frac{1}{N}\sum_{i=1}^{N} \sigma_{i}\left(u\left(X_{i}\right)-u_K\left(X_{i}\right)\right)\right]+\sum_{j=1}^{K-1} \mathbb{E}_{\{\sigma_i,X_i\}_{i=1}^N} \left[\sup _{u \in \mathcal{F}} \frac{1}{N}\sum_{i=1}^{N} \sigma_{i}\left(u_{j+1}\left(X_{i}\right)-u_j\left(X_{i}\right)\right)\right] \\
	&\quad+\mathbb{E}_{\{\sigma_i,X_i\}_{i=1}^N}\left[ \frac{1}{N}\sup _{u \in \mathcal{F}_1} \frac{1}{N}\sum_{i=1}^{N} \sigma_{i} u(X_i)\right].
\end{align*}
We can choose $\mathcal{F}_1=\{0\}$ to eliminate the third term. For the first term,
\begin{equation*}
	\mathbb{E}_{\{\sigma_i,X_i\}_{i=1}^N}\sup _{u \in \mathcal{F}} \frac{1}{N}\left[\sum_{i=1}^{N} \sigma_{i}\left(u\left(X_{i}\right)-u_K\left(X_{i}\right)\right)\right]
	\leq\mathbb{E}_{\{\sigma_i,X_i\}_{i=1}^N}\sup _{u \in \mathcal{F}} \frac{1}{N}\sum_{i=1}^{N} \left|\sigma_{i}\right|\left\|u-u_K\right\|_{\infty}\leq\epsilon_K.
\end{equation*}
For the second term, for any fixed samples $\{X_i\}_{i=1}^N$, we define
$$
V_j := \{(u_{j+1}\left(X_{1}\right)-u_j\left(X_{1}\right),\dots,u_{j+1}\left(X_{N}\right)-u_j\left(X_{N}\right)) \in \mathbb{R}^N: u\in\mathcal{F} \}.
$$
Then, for any $v^j\in V_j$,
\begin{align*}
	\|v^j\|_2&=\left(\sum_{i=1}^{n}\left|u_{j+1}(X_i)-u_j(X_i)\right|^2\right)^{1/2}
	\leq \sqrt{n}\left\|u_{j+1}-u_j\right\|_{\infty}\\
	&\leq\sqrt{n}\left\|u_{j+1}-u\right\|_{\infty}+\sqrt{n}\left\|u_{j}-u\right\|_{\infty}
	=\sqrt{n}\epsilon_{j+1}+\sqrt{n}\epsilon_{j}=3\sqrt{n}\epsilon_{j+1}.
\end{align*}
Applying Lemma $\ref{Mfinite}$, we have
\begin{align*}
	&\sum_{j=1}^{K-1} \mathbb{E}_{\{\sigma\}_{i=1}^N}\left[\sup _{u \in \mathcal{F}} \frac{1}{N}\sum_{i=1}^{N} \sigma_{i}\left(u_{j+1}\left(X_{i}\right)-u_j\left(X_{i}\right)\right)\right] \\
	&=\sum_{j=1}^{K-1} \mathbb{E}_{\{\sigma_i\}_{i=1}^N}\left[\sup _{v^j \in V_j} \frac{1}{N}\sum_{i=1}^{N} \sigma_{i}v_i^j\right]
	\leq\sum_{j=1}^{K-1}\frac{3\epsilon_{j+1}}{\sqrt{N}}\sqrt{2\log|V_j|}.
\end{align*}
By the definition of $V_j$, we know that $\left|V_j\right|\leq\left|\mathcal{F}_j\right|\left|\mathcal{F}_{j+1}\right|\leq\left|\mathcal{F}_{j+1}\right|^2$. Hence
\begin{align*}
	\sum_{j=1}^{K-1} \mathbb{E}_{\{\sigma_i,X_i\}_{i=1}^N} \left[\sup _{u \in \mathcal{F}} \frac{1}{N}\sum_{i=1}^{N} \sigma_{i}\left(u_{j+1}\left(X_{i}\right)-u_j\left(X_{i}\right)\right)\right] \leq\sum_{j=1}^{K-1}\frac{6\epsilon_{j+1}}{\sqrt{N}}\sqrt{\log \left|\mathcal{F}_{j+1}\right|}.
\end{align*}
Now we obtain
\begin{align*}
	\mathfrak{R}_N(\mathcal{F}) &
	\leq \epsilon_K+\sum_{j=1}^{K-1}\frac{6\epsilon_{j+1}}{\sqrt{N}}\sqrt{\log \left|\mathcal{F}_{j+1}\right|}\\
	&=\epsilon_K+\frac{12}{\sqrt{N}}\sum_{j=1}^{K-1}(\epsilon_{j+1}-\epsilon_{j+2})\sqrt{\log\mathcal{C}(\epsilon_{j+1},\mathcal{F},\|\cdot\|_{\infty})}\\
	&\leq\epsilon_{K}+\frac{12}{\sqrt{N}}\int_{\epsilon_{K+1}}^{B/2}\sqrt{\log\mathcal{C}(\epsilon,\mathcal{F},\|\cdot\|_{\infty})}d\epsilon.
\end{align*}
We conclude the lemma by choosing $K$ such that $\epsilon_{K+2}<\delta\leq\epsilon_{K+1}$ for any $0<\delta<B/2$.
\end{proof}

From Lemma $\ref{covering number Lipshcitz parameterization}$ we know that the key step to bound $\mathcal{C}(\epsilon,\mathcal{F}_i,\|\cdot\|_{\infty})$ with $\mathcal{F}_i$ defined in Lemma $\ref{symmetrization}$ is to compute the upper bound of Lipschitz constant of class $\mathcal{F}_i$, which is done in Lemma $\ref{f Lip}$-$\ref{Lip of Fi}$.
\begin{Lemma} \label{f Lip}
Let $\mathcal{D},\mathfrak{n}_{\mathcal{D}},n_i\in\mathbb{N}^+$, $n_\mathcal{D}=1$, $B_{\theta}\ge 1$ and $\rho$ be a bounded Lipschitz continuous function with $B_{\rho}, L_{\rho}\leq 1$. Assume that the parameterized function class
$
\mathcal{P}\subset\mathcal{N}_{\rho}\left(\mathcal{D}, \mathfrak{n}_{\mathcal{D}}, B_{\theta}\right)
$. For any $f(x;\theta)\in\mathcal{P}$, $f(x;\theta)$ is $\sqrt{\mathfrak{n}_{\mathcal{D}}}B_{\theta}^{\mathcal{D}-1}\left(\prod_{i=1}^{\mathcal{D}-1}n_i\right)$-Lipschitz continuous with respect to variable $\theta$, i.e.,
\begin{equation*}
\left|f(x;\theta)-{f}(x;\widetilde{\theta})\right|
\leq\sqrt{\mathfrak{n}_{\mathcal{D}}}B_{\theta}^{\mathcal{D}-1}\left(\prod_{i=1}^{\mathcal{D}-1}n_i\right)\left\|\theta-\widetilde{\theta}\right\|_2,\quad\forall x\in\Omega.
\end{equation*}
\end{Lemma}
\begin{proof}
For $\ell=2,\cdots,\mathcal{D}$(the argument for the case of $\ell=\mathcal{D}$ is slightly different),
\begin{align*}
	\left|f_q^{(\ell)}-\widetilde{f}_q^{(\ell)}\right|
	&=\left|\rho\left(\sum_{j=1}^{n_{\ell-1}}a_{qj}^{(\ell)}f_j^{(\ell-1)}+b_q^{(\ell)}\right)-\rho\left(\sum_{j=1}^{n_{\ell-1}}\widetilde{a}_{qj}^{(\ell)}\widetilde{f}_j^{(\ell-1)}+\widetilde{b}_q^{(\ell)}\right)\right|\\
	&\leq L_{\rho}\left|\sum_{j=1}^{n_{\ell-1}}a_{qj}^{(\ell)}f_j^{(\ell-1)}-\sum_{j=1}^{n_{\ell-1}}\widetilde{a}_{qj}^{(\ell)}\widetilde{f}_j^{(\ell-1)}+b_q^{(\ell)}-\widetilde{b}_q^{(\ell)}\right|\\
	&\leq L_{\rho}\sum_{j=1}^{n_{\ell-1}}\left|a_{qj}^{(\ell)}\right|\left|f_j^{(\ell-1)}-\widetilde{f}_j^{(\ell-1)}\right|+L_{\rho}\sum_{j=1}^{n_{\ell-1}}\left|a_{qj}^{(\ell)}-\widetilde{a}_{qj}^{(\ell)}\right|\left|\widetilde{f}_j^{(\ell-1)}\right|+L_{\rho}\left|b_q^{(\ell)}-\widetilde{b}_q^{(\ell)}\right|\\
	&\leq B_{\theta}L_{\rho}\sum_{j=1}^{n_{\ell-1}}\left|f_j^{(\ell-1)}-\widetilde{f}_j^{(\ell-1)}\right|+B_{\rho}L_{\rho}\sum_{j=1}^{n_{\ell-1}}\left|a_{qj}^{(\ell)}-\widetilde{a}_{qj}^{(\ell)}\right|+L_{\rho}\left|b_q^{(\ell)}-\widetilde{b}_q^{(\ell)}\right|\\
	&\leq B_{\theta}\sum_{j=1}^{n_{\ell-1}}\left|f_j^{(\ell-1)}-\widetilde{f}_j^{(\ell-1)}\right|+\sum_{j=1}^{n_{\ell-1}}\left|a_{qj}^{(\ell)}-\widetilde{a}_{qj}^{(\ell)}\right|+\left|b_q^{(\ell)}-\widetilde{b}_q^{(\ell)}\right|.
\end{align*}
For $\ell=1$,
\begin{align*}
	\left|f_q^{(1)}-\widetilde{f}_q^{(1)}\right|
	&= \left|\rho\left(\sum_{j=1}^{n_0}a_{qj}^{(1)}x_j^+b_q^{(1)}\right)-\rho\left(\sum_{j=1}^{n_0}\widetilde{a}_{qj}^{(1)}x_j+\widetilde{b}_q^{(1)}\right)\right|\\
	&\le \sum_{j=1}^{n_{0}}\left|a_{qj}^{(1)}-\widetilde{a}_{qj}^{(1)}\right|+\left|b_q^{(1)}-\widetilde{b}_q^{(1)}\right|
	=\sum_{j=1}^{\mathfrak{n}_1}\left|\theta_j-\widetilde{\theta}_j\right|.
\end{align*}
For $\ell=2$,
\begin{align*}
	\left|f_q^{(2)}-\widetilde{f}_q^{(2)}\right|
	&\leq B_{\theta}\sum_{j=1}^{n_{1}}\left|f_j^{(1)}-\widetilde{f}_j^{(1)}\right|+\sum_{j=1}^{n_{1}}\left|a_{qj}^{(2)}-\widetilde{a}_{qj}^{(2)}\right|+\left|b_q^{(2)}-\widetilde{b}_q^{(2)}\right|\\
	&\leq B_{\theta}\sum_{j=1}^{n_{1}}\sum_{k=1}^{\mathfrak{n}_1}\left|\theta_k-\widetilde{\theta}_k\right|+\sum_{j=1}^{n_{1}}\left|a_{qj}^{(2)}-\widetilde{a}_{qj}^{(2)}\right|+\left|b_q^{(2)}-\widetilde{b}_q^{(2)}\right|\\
	&\leq n_1B_{\theta}\sum_{j=1}^{\mathfrak{n}_2}\left|\theta_j-\widetilde{\theta}_j\right|.
\end{align*}
Assuming that for $\ell\geq2$,
\begin{align*}
\left|f_q^{(\ell)}-\widetilde{f}_q^{(\ell)}\right|
\leq \left(\prod_{i=1}^{\ell-1}n_i\right) B_{\theta}^{\ell-1}\sum_{j=1}^{\mathfrak{n}_{\ell}}\left|\theta_j-\widetilde{\theta}_j\right|,
\end{align*}
we have
\begin{align*}
	\left|f_q^{(\ell+1)}-\widetilde{f}_q^{(\ell+1)}\right|
	&\leq B_{\theta}\sum_{j=1}^{n_{\ell}}\left|f_j^{(\ell)}-\widetilde{f}_j^{(\ell)}\right|+\sum_{j=1}^{n_{\ell}}\left|a_{qj}^{(\ell+1)}-\widetilde{a}_{qj}^{(\ell+1)}\right|+\left|b_q^{(\ell+1)}-\widetilde{b}_q^{(\ell+1)}\right|\\
	&\leq B_{\theta}\sum_{j=1}^{n_{\ell}}\left(\prod_{i=1}^{\ell-1}n_i\right) B_{\theta}^{\ell-1}\sum_{k=1}^{\mathfrak{n}_1}\left|\theta_k-\widetilde{\theta}_k\right|+\sum_{j=1}^{n_{\ell}}\left|a_{qj}^{(\ell+1)}-\widetilde{a}_{qj}^{(\ell+1)}\right|+\left|b_q^{(\ell+1)}-\widetilde{b}_q^{(\ell+1)}\right|\\
	&\leq \left(\prod_{i=1}^{\ell}n_i\right) B_{\theta}^{\ell}\sum_{j=1}^{\mathfrak{n}_{\ell+1}}\left|\theta_j-\widetilde{\theta}_j\right|.
\end{align*}
Hence by induction and H$\mathrm{\ddot{o}}$lder inequality we conclude that
\begin{equation*}
\left|f-\widetilde{f}\right|
\leq \left(\prod_{i=1}^{\mathcal{D}-1}n_i\right) B_{\theta}^{\mathcal{D}-1}\sum_{j=1}^{\mathfrak{n}_{\mathcal{D}}}\left|\theta_j-\widetilde{\theta}_j\right|
\leq\sqrt{\mathfrak{n}_{\mathcal{D}}}B_{\theta}^{\mathcal{D}-1}\left(\prod_{i=1}^{\mathcal{D}-1}n_i\right)\left\|\theta-\widetilde{\theta}\right\|_2.
\end{equation*}
\end{proof}

\begin{Lemma}\label{f derivative bound}
Let $\mathcal{D},\mathfrak{n}_{\mathcal{D}},n_i\in\mathbb{N}^+$, $n_\mathcal{D}=1$, $B_{\theta}\ge 1$ and $\rho$ be a function such that $\rho'$ is bounded by $B_{\rho'}$. Assume that the parameterized function class
$
\mathcal{P}\subset\mathcal{N}_{\rho}\left(\mathcal{D}, \mathfrak{n}_{\mathcal{D}}, B_{\theta}\right)
$. Let $p=1,\cdots,d$. We have
\begin{align*}
	\left|\partial_{x_p}f_q^{(\ell)}\right|&\leq\left(\prod_{i=1}^{\ell-1}n_i\right)\left(B_{\theta}B_{\rho'}\right)^{\ell},
	\quad\ell=1,2,\cdots,\mathcal{D}-1,\\
	\left|\partial_{x_p}f\right|&\leq\left(\prod_{i=1}^{\mathcal{D}-1}n_i\right)B_{\theta}^{\mathcal{D}}B_{\rho'}^{\mathcal{D}-1}.
\end{align*}
\end{Lemma}
\begin{proof}
For $\ell=1,2,\cdots,\mathcal{D}-1$,
\begin{align*}
	\left|\partial_{x_p}f_q^{(\ell)}\right|
	&=\left|\sum_{j=1}^{n_{\ell-1}}a_{qj}^{(\ell)}\partial_{x_p}f_j^{(\ell-1)}\rho'\left(\sum_{j=1}^{n_{\ell-1}}a_{qj}^{(\ell)}f_j^{(\ell-1)}+b_q^{(\ell)}\right)\right|
	\leq B_{\theta}B_{\rho'}\sum_{j=1}^{n_{\ell-1}}\left|\partial_{x_p}f_j^{(\ell-1)}\right|\\
	&\leq\left(B_{\theta}B_{\rho'}\right)^2\sum_{k=1}^{n_{\ell-1}}\sum_{j=1}^{n_{\ell-2}}\left|\partial_{x_p}f_j^{(\ell-2)}\right|
	=n_{\ell-1}\left(B_{\theta}B_{\rho'}\right)^2\sum_{j=1}^{n_{\ell-2}}\left|\partial_{x_p}f_j^{(\ell-2)}\right|\\
	&\leq\cdots\leq\left(\prod_{i=2}^{\ell-1}n_i\right)\left(B_{\theta}B_{\rho'}\right)^{\ell-1}\sum_{j=1}^{n_{1}}\left|\partial_{x_p}f_j^{(1)}\right| \\
	&\leq\left(\prod_{i=2}^{\ell-1}n_i\right)\left(B_{\theta}B_{\rho'}\right)^{\ell-1}\sum_{j=1}^{n_{1}}B_{\theta}B_{\rho'}=\left(\prod_{i=1}^{\ell-1}n_i\right)\left(B_{\theta}B_{\rho'}\right)^{\ell}.
\end{align*}
The bound for $\left|\partial_{x_p}f\right|$ can be derived similarly.
\end{proof}

\begin{Lemma} \label{Df Lip}
Let $\mathcal{D},\mathfrak{n}_{\mathcal{D}},n_i\in\mathbb{N}^+$, $n_\mathcal{D}=1$, $B_{\theta}\ge 1$ and $\rho$ be a function such that $\rho,\rho'$ are bounded by $B_{\rho}, B_{\rho'}\le 1$ and have Lipschitz constants $L_{\rho}, L_{\rho'}\leq 1$, respectively. Assume that the parameterized function class
$
\mathcal{P}\subset\mathcal{N}_{\rho}\left(\mathcal{D}, \mathfrak{n}_{\mathcal{D}}, B_{\theta}\right)
$. Then, for any $f(x;\theta)\in\mathcal{P}$, $p=1,\cdots,d$,  $\partial_{x_p}f(x;\theta)$ is $\sqrt{\mathfrak{n}_{\mathcal{D}}}(\mathcal{D}+1)B_{\theta}^{2\mathcal{D}}\left(\prod_{i=1}^{\mathcal{D}-1}n_i\right)^2$-Lipschitz continuous with respect to variable $\theta$, i.e.,
\begin{equation*}
	\left|\partial_{x_p}f(x;\theta)-\partial_{x_p}f(x;\widetilde{\theta})\right|
	\leq \sqrt{\mathfrak{n}_{\mathcal{D}}}(\mathcal{D}+1)B_{\theta}^{2\mathcal{D}}\left(\prod_{i=1}^{\mathcal{D}-1}n_i\right)^2\left\|\theta-\widetilde{\theta}\right\|_2,\quad \forall x\in\Omega.
\end{equation*}
\end{Lemma}
\begin{proof}
For $\ell=1$,
\begin{align*}
	&\left|\partial_{x_p}f_q^{(1)}-\partial_{x_p}\widetilde{f}_q^{(1)}\right| \\
	=&\left|a_{qp}^{(1)}\rho'\left(\sum_{j=1}^{n_{0}}a_{qj}^{(1)}x_j+b_q^{(1)}\right)-\widetilde{a}_{qp}^{(1)}\rho'\left(\sum_{j=1}^{n_{0}}\widetilde{a}_{qj}^{(1)}x_j+\widetilde{b}_q^{(1)}\right)\right|\\
	\leq&\left|a_{qp}^{(1)}-\widetilde{a}_{qp}^{(1)}\right|\left|\rho'\left(\sum_{j=1}^{n_{0}}a_{qj}^{(1)}x_j+b_q^{(1)}\right)\right|
	+\left|\widetilde{a}_{qp}^{(1)}\right|\left|\rho'\left(\sum_{j=1}^{n_{0}}a_{qj}^{(1)}x_j+b_q^{(1)}\right)-\rho'\left(\sum_{j=1}^{n_{0}}\widetilde{a}_{qj}^{(1)}x_j+\widetilde{b}_q^{(1)}\right)\right|\\
	\leq& B_{\rho'}\left|a_{qp}^{(1)}-\widetilde{a}_{qp}^{(1)}\right|+B_{\theta}L_{\rho'}\sum_{j=1}^{n_{0}}\left|a_{qj}^{(1)}-\widetilde{a}_{qj}^{(1)}\right|+B_{\theta}L_{\rho'}\left|{b}_q^{(1)}-\widetilde{b}_q^{(1)}\right|\leq 2B_{\theta}\sum_{k=1}^{\mathfrak{n}_{1}}\left|\theta_k-\widetilde{\theta}_k\right|
\end{align*}
For $\ell\geq 2$, we establish the Recurrence relation:
\begin{align*}
	&\left|\partial_{x_p}f_q^{(\ell)}-\partial_{x_p}\widetilde{f}_q^{(\ell)}\right|\\
	&\leq\sum_{j=1}^{n_{\ell-1}}\left|a_{qj}^{(\ell)}\right|\left|\partial_{x_p}f_j^{(\ell-1)}\right|\left|\rho'\left(\sum_{j=1}^{n_{\ell-1}}a_{qj}^{(\ell)}f_j^{(\ell-1)}+b_q^{(\ell)}\right)-\rho'\left(\sum_{j=1}^{n_{\ell-1}}\widetilde{a}_{qj}^{(\ell)}\widetilde{f}_j^{(\ell-1)}+\widetilde{b}_q^{(\ell)}\right)\right|\\
	&\quad+\sum_{j=1}^{n_{\ell-1}}\left|a_{qj}^{(\ell)}\partial_{x_p}f_j^{(\ell-1)}-\widetilde{a}_{qj}^{(\ell)}\partial_{x_p}\widetilde{f}_j^{(\ell-1)}\right|\left|\rho'\left(\sum_{j=1}^{n_{\ell-1}}\widetilde{a}_{qj}^{(\ell)}\widetilde{f}_j^{(\ell-1)}+\widetilde{b}_q^{(\ell)}\right)\right|\\
	&\leq B_{\theta}L_{\rho'}\sum_{j=1}^{n_{\ell-1}}\left|\partial_{x_p}f_j^{(\ell-1)}\right|\left(\sum_{j=1}^{n_{\ell-1}}\left|a_{qj}^{(\ell)}f_j^{(\ell-1)}-\widetilde{a}_{qj}^{(\ell)}\widetilde{f}_j^{(\ell-1)}\right|+\left|b_q^{(\ell)}-\widetilde{b}_q^{(\ell)}\right|\right)\\
	&\quad+B_{\rho'}\sum_{j=1}^{n_{\ell-1}}\left|a_{qj}^{(\ell)}\partial_{x_p}f_j^{(\ell-1)}-\widetilde{a}_{qj}^{(\ell)}\partial_{x_p}\widetilde{f}_j^{(\ell-1)}\right|\\
	&\leq B_{\theta}L_{\rho'}\sum_{j=1}^{n_{\ell-1}}\left|\partial_{x_p}f_j^{(\ell-1)}\right|\left(B_{\rho}\sum_{j=1}^{n_{\ell-1}}\left|a_{qj}^{(\ell)}-\widetilde{a}_{qj}^{(\ell)}\right|+B_{\theta}\sum_{j=1}^{n_{\ell-1}}\left|f_j^{(\ell-1)}-\widetilde{f}_j^{(\ell-1)}\right|+\left|b_q^{(\ell)}-\widetilde{b}_q^{(\ell)}\right|\right)\\
	&\quad+B_{\rho'}B_{\theta}\sum_{j=1}^{n_{\ell-1}}\left|\partial_{x_p}f_j^{(\ell-1)}-\partial_{x_p}\widetilde{f}_j^{(\ell-1)}\right|+B_{\rho'}\sum_{j=1}^{n_{\ell-1}}\left|a_{qj}^{(\ell)}-\widetilde{a}_{qj}^{(\ell)}\right|\left|\partial_{x_p}\widetilde{f}_j^{(\ell-1)}\right|\\
	&\leq B_{\theta}\sum_{j=1}^{n_{\ell-1}}\left|\partial_{x_p}f_j^{(\ell-1)}\right|\left(\sum_{j=1}^{n_{\ell-1}}\left|a_{qj}^{(\ell)}-\widetilde{a}_{qj}^{(\ell)}\right|+B_{\theta}\sum_{j=1}^{n_{\ell-1}}\left|f_j^{(\ell-1)}-\widetilde{f}_j^{(\ell-1)}\right|+\left|b_q^{(\ell)}-\widetilde{b}_q^{(\ell)}\right|\right)\\
	&\quad+B_{\theta}\sum_{j=1}^{n_{\ell-1}}\left|\partial_{x_p}f_j^{(\ell-1)}-\partial_{x_p}\widetilde{f}_j^{(\ell-1)}\right|+\sum_{j=1}^{n_{\ell-1}}\left|a_{qj}^{(\ell)}-\widetilde{a}_{qj}^{(\ell)}\right|\left|\partial_{x_p}\widetilde{f}_j^{(\ell-1)}\right|\\
	&\leq B_{\theta}\left(\prod_{i=1}^{\ell-1}n_i\right)B_{\theta}^{\ell}\left(\sum_{j=1}^{n_{\ell-1}}\left|a_{qj}^{(\ell)}-\widetilde{a}_{qj}^{(\ell)}\right|+B_{\theta}\sum_{j=1}^{n_{\ell-1}}
	\left(\prod_{i=1}^{\ell-2}n_i\right) B_{\theta}^{\ell-2}\sum_{k=1}^{\mathfrak{n}_{\ell-1}}\left|\theta_k-\widetilde{\theta}_k\right|+\left|b_q^{(\ell)}-\widetilde{b}_q^{(\ell)}\right|\right)\\
	&\quad+B_{\theta}\sum_{j=1}^{n_{\ell-1}}\left|\partial_{x_p}f_j^{(\ell-1)}-\partial_{x_p}\widetilde{f}_j^{(\ell-1)}\right|+\sum_{j=1}^{n_{\ell-1}}\left|a_{qj}^{(\ell)}-\widetilde{a}_{qj}^{(\ell)}\right|\left(\prod_{i=1}^{\ell-2}n_i\right)B_{\theta}^{\ell-1}\\
	&\leq B_{\theta}\sum_{j=1}^{n_{\ell-1}}\left|\partial_{x_p}f_j^{(\ell-1)}-\partial_{x_p}\widetilde{f}_j^{(\ell-1)}\right|
	+B_{\theta}^{2\ell}\left(\prod_{i=1}^{\ell-1}n_i\right)^2\sum_{k=1}^{\mathfrak{n}_{\ell}}\left|\theta_k-\widetilde{\theta}_k\right|
\end{align*}
For $\ell=2$,
\begin{align*}
	\left|\partial_{x_p}f_q^{(2)}-\partial_{x_p}\widetilde{f}_q^{(2)}\right|&\leq B_{\theta}\sum_{j=1}^{n_{1}}\left|\partial_{x_p}f_j^{(1)}-\partial_{x_p}\widetilde{f}_j^{(1)}\right|
	+B_{\theta}^{4}n_1^2\sum_{k=1}^{\mathfrak{n}_{2}}\left|\theta_k-\widetilde{\theta}_k\right|\\
	&\leq2B_{\theta}^2n_1\sum_{k=1}^{\mathfrak{n}_{1}}\left|\theta_k-\widetilde{\theta}_k\right|+B_{\theta}^{4}n_1^2\sum_{k=1}^{\mathfrak{n}_{2}}\left|\theta_k-\widetilde{\theta}_k\right|
	\leq 3B_{\theta}^{4}n_1^2\sum_{k=1}^{\mathfrak{n}_{2}}\left|\theta_k-\widetilde{\theta}_k\right|
\end{align*}
Assuming that for $\ell\geq2$,
\begin{equation*}
	\left|\partial_{x_p}f_q^{(\ell)}-\partial_{x_p}\widetilde{f}_q^{(\ell)}\right|\leq (\ell+1) B_{\theta}^{2\ell}\left(\prod_{i=1}^{\ell-1}n_i\right)^2\sum_{k=1}^{\mathfrak{n}_{\ell}}\left|\theta_k-\widetilde{\theta}_k\right|
\end{equation*}
we have
\begin{align*}
	&\left|\partial_{x_p}f_q^{(\ell+1)}-\partial_{x_p}\widetilde{f}_q^{(\ell+1)}\right| \\
	\leq& B_{\theta}\sum_{j=1}^{n_{\ell}}\left|\partial_{x_p}f_j^{(\ell)}-\partial_{x_p}\widetilde{f}_j^{(\ell)}\right|
	+B_{\theta}^{2\ell+2}\left(\prod_{i=1}^{\ell}n_i\right)^2\sum_{k=1}^{\mathfrak{n}_{\ell+1}}\left|\theta_k-\widetilde{\theta}_k\right|\\
	\leq& B_{\theta}\sum_{j=1}^{n_{\ell}}(\ell+1) B_{\theta}^{2\ell}\left(\prod_{i=1}^{\ell-1}n_i\right)^2\sum_{k=1}^{\mathfrak{n}_{\ell}}\left|\theta_k-\widetilde{\theta}_k\right|
	+B_{\theta}^{2\ell+2}\left(\prod_{i=1}^{\ell}n_i\right)^2\sum_{k=1}^{\mathfrak{n}_{\ell+1}}\left|\theta_k-\widetilde{\theta}_k\right|\\
	\leq& (\ell+2)B_{\theta}^{2\ell+2}\left(\prod_{i=1}^{\ell}n_i\right)^2\sum_{k=1}^{\mathfrak{n}_{\ell+1}}\left|\theta_k-\widetilde{\theta}_k\right|
\end{align*}
Hence by by induction and H$\mathrm{\ddot{o}}$lder inequality we conclude that
\begin{equation*}
\left|\partial_{x_p}f-\partial_{x_p}\widetilde{f}\right|
\leq (\mathcal{D}+1)B_{\theta}^{2\mathcal{D}}\left(\prod_{i=1}^{\mathcal{D}-1}n_i\right)^2\sum_{k=1}^{\mathfrak{n}_{\mathcal{D}}}\left|\theta_k-\widetilde{\theta}_k\right|
\leq \sqrt{\mathfrak{n}_{\mathcal{D}}}(\mathcal{D}+1)B_{\theta}^{2\mathcal{D}}\left(\prod_{i=1}^{\mathcal{D}-1}n_i\right)^2\left\|\theta-\widetilde{\theta}\right\|_2
\end{equation*}
\end{proof}

\begin{Lemma} \label{Lip of Fi}
Let $\mathcal{D},\mathfrak{n}_{\mathcal{D}},n_i\in\mathbb{N}^+$, $n_\mathcal{D}=1$, $B_{\theta}\ge 1$ and $\rho$ be a function such that $\rho,\rho'$ are bounded by $B_{\rho}, B_{\rho'}\le 1$ and have Lipschitz constants $L_{\rho}, L_{\rho'}\leq 1$, respectively. Assume that the parameterized function class
$\mathcal{P}\subset\mathcal{N}_{\rho}\left(\mathcal{D}, \mathfrak{n}_{\mathcal{D}}, B_{\theta}\right)$. Then $\mathcal{F}_1,\mathcal{F}_2,\mathcal{F}_3,\mathcal{F}_5$ are parameterized function classes with respect to parameter set $\Theta\times\Theta$ and $\mathcal{F}_4,\mathcal{F}_6$ are parameterized function classes with respect to parameter set $\Theta$ with $\Theta:=\{\theta\in\mathbb{R}^{\mathfrak{n}_{\mathcal{D}}}:\|\theta\|_2\leq B_{\theta}\}$. In addition, for any $f_i(x;\theta),f_i(x;\widetilde{\theta})\in\mathcal{F}_i$, $i=1,\cdots,6$, we have
\begin{align*}
|f_i(x;\theta)| &\le B_i, \quad \forall x\in\Omega,\\
|f_i(x;\theta)-f_i(x;\widetilde{\theta})| &\leq L_i\|\theta-\widetilde{\theta}\|_2,\quad \forall x\in\Omega,
\end{align*}
with
\begin{align*}
B_1&=C(coe)d^2B_{\theta}^{2\mathcal{D}}\left(\prod_{i=1}^{\mathcal{D}-1}n_i\right)^2,
&&B_2=C(coe)d(n_{\mathcal{D}-1}+1)B_{\theta}^{\mathcal{D}+1}\left(\prod_{i=1}^{\mathcal{D}-1}n_i\right),\\
B_3&=C(coe)(n_{\mathcal{D}-1}+1)^2B_{\theta}^{2},
&&B_4=C(coe)(n_{\mathcal{D}-1}+1)B_{\theta},\\
B_5&=\frac{\alpha}{2}(n_{\mathcal{D}-1}+1)^2B_{\theta}^{2},
&&B_6=C(coe)(n_{\mathcal{D}-1}+1)B_{\theta}
\end{align*}
and
\begin{align*}
L_1&=C(coe)d^2\sqrt{2\mathfrak{n}_{\mathcal{D}}}(\mathcal{D}+1)B_{\theta}^{3\mathcal{D}}\left(\prod_{i=1}^{\mathcal{D}-1}n_i\right)^3,\\
L_2&=C(coe)d\sqrt{2\mathfrak{n}_{\mathcal{D}}}(\mathcal{D}+1)(n_{\mathcal{D}-1}+1)B_{\theta}^{2\mathcal{D}+1}\left(\prod_{i=1}^{\mathcal{D}-1}n_i\right)^2,\\
L_3&=C(coe)\sqrt{2\mathfrak{n}_{\mathcal{D}}}(n_{\mathcal{D}-1}+1)B_{\theta}^{\mathcal{D}}\left(\prod_{i=1}^{\mathcal{D}-1}n_i\right),\\
L_4&=C(coe)\sqrt{\mathfrak{n}_{\mathcal{D}}}B_{\theta}^{\mathcal{D}-1}\left(\prod_{i=1}^{\mathcal{D}-1}n_i\right),\\
L_5&=\frac{\alpha}{2}\sqrt{2\mathfrak{n}_{\mathcal{D}}}(n_{\mathcal{D}-1}+1)B_{\theta}^{\mathcal{D}}\left(\prod_{i=1}^{\mathcal{D}-1}n_i\right),\\
L_6&=C(coe)\sqrt{\mathfrak{n}_{\mathcal{D}}}B_{\theta}^{\mathcal{D}-1}\left(\prod_{i=1}^{\mathcal{D}-1}n_i\right)
\end{align*}
\end{Lemma}
\begin{proof}
A direct result from Lemmas $\ref{f Lip}$, $\ref{f derivative bound}$, $\ref{Df Lip}$, and standard calculation.
\end{proof}

Now we state our main result with respect to statistical error $\mathcal{E}_{sta}$.
\begin{Theorem} \label{sta error}
Let $\mathcal{D},\mathfrak{n}_{\mathcal{D}},n_i\in\mathbb{N}^+$, $n_\mathcal{D}=1$, $B_{\theta}\ge 1$ and $\rho$ be a function such that $\rho,\rho'$ are bounded by $B_{\rho}, B_{\rho'}\le 1$ and have Lipschitz constants $L_{\rho}, L_{\rho'}\leq 1$, respectively. Assume that the parameterized function class
$
	\mathcal{P}\subset\mathcal{N}_{\rho}\left(\mathcal{D}, \mathfrak{n}_{\mathcal{D}}, B_{\theta}\right)
$. Then, if $N=M$, we have for $\mathcal{E}_{sta}$ defined in $(\ref{sta})$,
\begin{equation*}
\mathbb{E}_{\{{X_i}\}_{i=1}^{N},\{{Y_j}\}_{j=1}^{M}}\mathcal{E}_{sta}
\leq \frac{C(\Omega,coe,\alpha)}{\beta}\frac{d^3\mathcal{D}^{\frac{1}{2}}\mathfrak{n}_{\mathcal{D}}^{\frac{7}{2}\mathcal{D}-\frac{3}{2}}B_{\theta}^{\frac{7}{2}\mathcal{D}+\frac{1}{2}}}{N^{\frac{1}{4}}}
\end{equation*}
\end{Theorem}
\begin{proof}
From Lemma $\ref{covering number Euclidean space}$, $\ref{covering number Lipshcitz parameterization}$ and $\ref{chaining}$, we have
\begin{align*}
	\mathfrak{R}_N(\mathcal{F}_i)&\leq\inf_{0<\delta<B_i/2}\left(4\delta+\frac{12}{\sqrt{N}}\int_{\delta}^{B_i/2}\sqrt{\log\mathcal{C}(\epsilon,\mathcal{F}_i,\|\cdot\|_{\infty})}d\epsilon\right)\\
	&\leq\inf_{0<\delta<B_i/2}\left(4\delta+\frac{12}{\sqrt{N}}\int_{\delta}^{B_i/2}\sqrt{\mathfrak{n}_{\mathcal{D}}\log\left(\frac{2L_iB_\theta\sqrt{\mathfrak{n}_{\mathcal{D}}}}{\epsilon}\right)}d\epsilon\right) \\
	&\leq\inf_{0<\delta<B_i/2}\left(4\delta+\frac{6\sqrt{\mathfrak{n}_{\mathcal{D}}}B_i}{\sqrt{N}}\sqrt{\log\left(\frac{2L_iB_\theta\sqrt{\mathfrak{n}_{\mathcal{D}}}}{\delta}\right)}\right).
\end{align*}
Choosing $\delta=1/\sqrt{N}<B_i/2$ and applying Lemma $\ref{Lip of Fi}$, we have for $i=1,\cdots,4$,
\begin{align}
\mathfrak{R}_N(\mathcal{F}_i)&\leq\frac{4}{\sqrt{N}}+\frac{6\sqrt{\mathfrak{n}_{\mathcal{D}}}B_i}{\sqrt{N}}\sqrt{\log\left(2L_iB_\theta\sqrt{\mathfrak{n}_{\mathcal{D}}}\sqrt{N}\right)}\nonumber\\
&\leq \frac{C(coe,\alpha)}{\sqrt{N}}\cdot d^2\sqrt{\mathfrak{n}_{\mathcal{D}}}\left(\prod_{i=1}^{\mathcal{D}-1}n_i\right)^2B_{\theta}^{2\mathcal{D}}\sqrt{\log\left(d^2\mathfrak{n}_{\mathcal{D}}(\mathcal{D}+1)B_{\theta}^{3\mathcal{D}+1}\left(\prod_{i=1}^{\mathcal{D}-1}n_i\right)^3\sqrt{N}\right)} \nonumber\\
&\leq \frac{C(coe,\alpha)}{\sqrt{N}}\cdot d^2\mathfrak{n}_{\mathcal{D}}^{2\mathcal{D}-\frac{1}{2}}B_{\theta}^{2\mathcal{D}} \sqrt{\log\left(d^2\mathcal{D}\mathfrak{n}_{\mathcal{D}}^{3\mathcal{D}-2}B_{\theta}^{3\mathcal{D}+1}\sqrt{N}\right)}\nonumber\\
&\leq\frac{C(coe,\alpha)d^3\mathcal{D}^{\frac{1}{2}}\mathfrak{n}_{\mathcal{D}}^{\frac{7}{2}\mathcal{D}-\frac{3}{2}}B_{\theta}^{\frac{7}{2}\mathcal{D}+\frac{1}{2}}}{N^{\frac{1}{4}}}\label{sta error1}
\end{align}
Similarly, for $i=5,6$,
\begin{align}
\mathfrak{R}_M(\mathcal{F}_i)\leq\frac{C(coe,\alpha)d^3\mathcal{D}^{\frac{1}{2}}\mathfrak{n}_{\mathcal{D}}^{\frac{7}{2}\mathcal{D}-\frac{3}{2}}B_{\theta}^{\frac{7}{2}\mathcal{D}+\frac{1}{2}}}{M^{\frac{1}{4}}}\label{sta error2}
\end{align}
Combining Lemma $\ref{triangle inequality}$, $\ref{symmetrization}$, $(\ref{sta error1})$ and $(\ref{sta error2})$, we obtain, if $N=M$,
\begin{equation*}
\mathbb{E}_{\{{X_i}\}_{i=1}^{N},\{{Y_j}\}_{j=1}^{M}}\mathcal{E}_{sta}
\leq \frac{C(\Omega,coe,\alpha)}{\beta}\frac{d^3\mathcal{D}^{\frac{1}{2}}\mathfrak{n}_{\mathcal{D}}^{\frac{7}{2}\mathcal{D}-\frac{3}{2}}B_{\theta}^{\frac{7}{2}\mathcal{D}+\frac{1}{2}}}{N^{\frac{1}{4}}}
\end{equation*}

\end{proof}

\section{Covergence Rate for the Galerkin Method}\label{sec:rate}

Now we state our main result.
\begin{Theorem}
Let (A1)-(A3) holds. Assume that $\mathcal{E}_{opt}=0$. Let $\rho$ be logistic function $\frac{1}{1+e^{-x}}$ or tanh function $\frac{e^x-e^{-x}}{e^x+e^{-x}}$. Let $u_{\phi_{\mathcal{A}}}$ be the solution of problem $(\ref{optimization})$ generated by a random solver $\mathcal{A}$.

(1)Let $u_R$ be the weak solution of Robin problem $(\ref{second order elliptic equation})(\ref{robin})$. Assume that $\epsilon>0$ is sufficiently small. Set
the parameterized function class
\begin{equation*}
\mathcal{P}:=\mathcal{N}_{\rho}\left(C\log(d+1),C(d,coe,\beta)\epsilon^{\frac{-d}{1-\mu}},C(d,coe,\beta) \epsilon^{\frac{-9d-8}{2-2\mu}}\right)\bigcap B_{H^1(\Omega)}(0,2)
\end{equation*}
where $\mu>0$ can be any arbitrarily small number and  $B_{H^1(\Omega)}(0,2):=\{f\in H^1(\Omega):\|f\|_{H^1(\Omega)}\leq2\}$. Set the number of samples
\begin{equation*}
N=M=C(d,\Omega, coe,\alpha,\beta)\epsilon^{-Cd\log d},
\end{equation*}
then
\begin{equation*}
\mathbb{E}_{\{{X_i}\}_{i=1}^{N},\{{Y_j}\}_{j=1}^{M}}\|u_{\phi_\mathcal{A}}-u_R\|_{H^1(\Omega)}\leq \epsilon.
\end{equation*}

(2)Let $u_D$ be the weak solution of Dirichlet problem $(\ref{second order elliptic equation})(\ref{dirichlet})$. Set $\alpha=1,g=0$. Assume that $\epsilon>0$ is sufficiently small. Set the penalty parameter $\beta=C(d,\Omega,coe)\epsilon^2$. Set the parameterized function class
\begin{equation*}
\mathcal{P}:=\mathcal{N}_{\rho}\left(C\log(d+1),C(d,coe)\epsilon^{\frac{-d}{1-\mu}},C(d,coe) \epsilon^{\frac{-9d-8}{2-2\mu}}\right)\bigcap B_{H^1(\Omega)}(0,2)
\end{equation*}
where $\mu>0$ can be any arbitrarily small number and $B_{H^1(\Omega)}(0,2):=\{f\in H^1(\Omega):\|f\|_{H^1(\Omega)}\leq2\}$. Set the number of samples
\begin{equation*}
N=M=C(d,\Omega, coe)\epsilon^{-Cd\log d},
\end{equation*}
then
\begin{equation*}
\mathbb{E}_{\{{X_i}\}_{i=1}^{N},\{{Y_j}\}_{j=1}^{M}}\|u_{\phi_\mathcal{A}}-u_D\|_{H^1(\Omega)}\leq \epsilon.
\end{equation*}
\end{Theorem}

\begin{proof}
Setting the approximation error $\mathcal{E}_{app}$ as $\frac{\epsilon}{2}$ in Theorem $\ref{app error}$ and the statisitical error $\mathcal{E}_{sta}$ as $\frac{\epsilon}{2}$ in Theorem $\ref{sta error}$. Combining Proposition $\ref{error decomposition}$, Theorem $\ref{app error}$ and Theorem $\ref{sta error}$ yields (1).

Setting the approximation error $\mathcal{E}_{app}$ as $\frac{\epsilon}{3}$ in Theorem $\ref{app error}$, the statisitical error $\mathcal{E}_{sta}$ as $\frac{\epsilon}{3}$ in Theorem $\ref{sta error}$ and the penalty error $\mathcal{E}_{pen}$ as $\frac{\epsilon}{3}$ in Lemma $\ref{penalty convergence}$. Combining Proposition $\ref{error decomposition}$, Theorem $\ref{app error}$, Theorem $\ref{sta error}$ and Lemma $\ref{penalty convergence}$ yields (2).
\end{proof}

\section{Conclusions and Extensions}\label{sec:conclusion}

This paper analyzes the convergence rate of the deep Galerkin method (DGMW) for second-order elliptic equations in $\mathbb{R}^d$ with Dirichlet, Neumann, and Robin boundary conditions, respectively.   We provide the first  $\mathcal{O}(n^{-1/d})$ convergence rate of DGMW by properly choosing the depth and width of the two networks in terms of the number of training samples $n$.
We will extend the current analysis to the Friedrichs learning method \cite{chen2020friedrichs} in our future work.

\section*{Acknowledgements}	
	
The work of Y. Jiao is supported in part by the National Science Foundation of China under Grant 11871474 and by the research
fund of KLATASDSMOE. The work of Y. Wang is supported in part by the Hong Kong Research Grant Council grants 16308518 and
16317416 and  HK Innovation Technology Fund ITS/044/18FX, as well as Guangdong-Hong Kong-Macao Joint Laboratory for Data-Driven Fluid Mechanics and Engineering Applications. The work of H. Yang was partially supported by the US National Science Foundation under award DMS-2244988, DMS-2206333, and the Office of Naval Research Award N00014-23-1-2007.

\bibliographystyle{siam}
\bibliography{ref}	
\end{document}